\documentclass[pdflatex,12pt,a4paper]{amsart}
\usepackage[a4paper,left=30mm,right=30mm,top=36mm,bottom=36mm]{geometry}
\usepackage{amssymb}
\usepackage{graphicx}
\usepackage{epsfig}
\usepackage{nicefrac}
\usepackage{mathrsfs}
\usepackage{float}
\usepackage[utf8]{inputenc}
\usepackage[english]{babel}
\usepackage{pdfpages}
\usepackage{bbm}
\usepackage{esint}
\usepackage{blindtext}
\usepackage{subfigure}
\newcommand{\R}{\mathbb{R}}
\newcommand{\N}{\mathbb{N}}

\newcommand{\eps}{\varepsilon}
\newcommand{\fhi}{\varphi}
\newcommand{\weak}{\rightharpoonup}

\newcommand{\weakstar}{\stackrel{\ast}{\rightharpoonup}}

\renewcommand{\div}{\mathrm{div}}

\newcommand{\supp}{\mathrm{supp}}

\def\calG{\mathcal{G}}

\def\calO{\mathcal{O}}
\def\calH{\mathcal{H}}

\def\calT{\mathcal{T}}

\newtheorem{definition}{Definition}[section] 
\newtheorem{remark}{Remark}[section]    
\newtheorem{theorem}{Theorem}[section]    
   
\newtheorem{lemma}{Lemma}[section]

\def\XXint#1#2#3{{\setbox0=\hbox{$#1{#2#3}{\int}$}
     \vcenter{\hbox{$#2#3$}}\kern-.5\wd0}}

\numberwithin{equation}{section}
\usepackage{scrpage2}

\begin{document}

\title[Approximation of Nondivergence-Form Homogenization Problems]{Finite Element Approximation of Elliptic Homogenization Problems in Nondivergence-Form}

\author[Y. Capdeboscq]{Yves Capdeboscq$^{\ast}$}
\address[*]{Universit\'e de Paris,  CNRS, Sorbonne Universit\'e, Laboratoire Jacques-Louis Lions UMR7598, Paris, France}
\email[Y. Capdeboscq]{yves.capdeboscq@sorbonne-universite.fr}

\author[T. Sprekeler]{Timo Sprekeler$^{\dagger}$}
\address[$\dagger$]{University of Oxford, Mathematical Institute, Woodstock Road, Oxford OX2 6GG, UK. }
\email[T. Sprekeler]{timo.sprekeler@maths.ox.ac.uk}

\author[E. Süli]{Endre Süli$^{\ddagger}$}
\address[$\ddagger$]{University of Oxford, Mathematical Institute, Woodstock Road, Oxford OX2 6GG, UK. }
\email[E. Süli]{endre.suli@maths.ox.ac.uk}

\date{May 28, 2019}

\begin{abstract}
We use uniform $W^{2,p}$ estimates to obtain corrector results for periodic homogenization problems of the form $A(x/\eps):D^2 u_{\eps} = f$ subject to a homogeneous Dirichlet boundary condition. We propose and rigorously analyze a numerical scheme based on finite element approximations for such nondivergence-form homogenization problems. The second part of the paper focuses on the approximation of the corrector and numerical homogenization for the case of nonuniformly oscillating coefficients. Numerical experiments demonstrate the performance of the scheme.
\end{abstract}

\maketitle

\tableofcontents

\section{Introduction}

In this work we consider second-order elliptic equations of nondivergence structure, involving rapidly oscillating coefficients, of the form
\begin{align}\label{Intro1}
A\left(\frac{\cdot}{\eps}\right):D^2 u_{\eps}:=\sum_{i,j=1}^n a_{ij}\left(\frac{\cdot}{\eps}\right)\partial_{ij}^2 u_{\eps}&=f\quad\text{in }\Omega,
\end{align}
subject to the homogeneous Dirichlet boundary condition
\begin{align}\label{Intro2}
u_{\eps}=0\quad\text{on }\partial\Omega.
\end{align}
Here we assume that $\Omega\subset \R^n$ is a sufficiently regular bounded domain, $\eps>0$ is small, and that $A=(a_{ij}):\R^n\rightarrow \R^{n\times n}$ is a symmetric, uniformly elliptic and $(0,1)^n$-periodic matrix-valued function such that
\begin{align*}
A\in W^{1,q}(Y)\text{ for some }q>n,
\end{align*}
where $Y:=(0,1)^n$ denotes the unit cell, see \eqref{Aass}. The main goal of this paper is to propose and analyze a numerical homogenization scheme  for \eqref{Intro1}, \eqref{Intro2} that is based on finite element approximations.

The theory of periodic homogenization is concerned with the limiting behavior of the solutions as the oscillation parameter $\eps$ tends to zero. For the problem \eqref{Intro1}, \eqref{Intro2} under consideration a classical homogenization theorem (see \cite[Sec. 3, Theorem 5.2]{BLP11}) states that the solution sequence $(u_{\eps})_{\eps>0}$ converges in an appropriate Sobolev space to the solution $u_0$ to the problem 
\begin{align}\label{Intro25}
\left\{ \begin{aligned}A^0:D^2 u_0 &= f\quad \text{in }\Omega,\\ \hfill u_0&= 0\quad \text{on }\partial\Omega.\end{aligned}\right.
\end{align} 
Here $A^0\in \R^{n\times n}$ is the constant matrix given by
\begin{align}\label{Intro26}
A^0=\int_Y Am,
\end{align}
and $m:\R^n\rightarrow \R$ is the invariant measure, i.e. the solution to the problem
\begin{align*}
\begin{cases}
D^2:(Am)=0\quad\text{in }Y,\\
m\;\text{is $Y$-periodic},\; \int_Y m = 1;
\end{cases}
\end{align*}
see Section 2 for further details. The task of numerical homogenization is the numerical approximation of the matrix $A^0$ and the solution $u_0$ to the homogenized problem \eqref{Intro25}. As it turns out, $u_0$ provides a good approximation to $u_{\eps}$ in $H^1(\Omega)$, and by adding corrector terms it is possible to obtain an $H^2(\Omega)$-norm approximation. Note that the approximation of \eqref{Intro1}, \eqref{Intro2} by a standard $H^2(\Omega)$-conforming finite element method does not yield error bounds independent of $\eps$, since for $s>0$ one has that
\begin{align*}
\|u_{\eps}\|_{H^{2+s}(\Omega)}=\mathcal{O}\left(\eps^{-s}\right).
\end{align*}

The motivation for investigating second-order elliptic problems in nondivergence-form comes from physics, engineering, as well as mathematical areas such as stochastic analysis. A notable example of a nonlinear PDE of nondivergence structure is the Hamilton--Jacobi--Bellman equation, which arises in stochastic control theory. The asymptotic behavior of PDEs with rapidly oscillating coefficients is also of importance when micro-inhomogeneous media are investigated.

Over the past decades significant work has been done on periodic homogenization of elliptic problems in divergence-form; numerical homogenization for nondivergence-form problems is however less developed.

The theory of homogenization of divergence-form problems such as
\begin{align}\label{Intro4}
\nabla\cdot \left(A\left(\frac{\cdot}{\eps}\right)\nabla u_{\eps}\right) + b\left(\frac{\cdot}{\eps}\right)\cdot \nabla u_{\eps} = f\quad\text{in }\Omega
\end{align}
with periodic and sufficiently regular $A:\R^n\rightarrow \R^{n\times n}$ and $b:\R^n\rightarrow \R^n$ is extensively covered in the books \cite{All02,BLP11,CD99,Tar09}. For divergence-form problems, various multiscale finite element methods (MsFEM) have been developed, which have the advantage over classical finite element methods of providing accurate approximations for very small values of $\eps$ even for moderate values of the grid size. The book \cite{EH09} by Efendiev and Hou contains a detailed overview of these methods.

It is important to note that although, if $A$ is sufficiently smooth, equation \eqref{Intro1} can be rewritten in divergence-form,  
\begin{align}\label{Intro3}
\nabla\cdot \left(A\left(\frac{\cdot}{\eps}\right)\nabla u_{\eps}\right) - \frac{1}{\eps}\left(\div A\right)\left(\frac{\cdot}{\eps}\right)\cdot \nabla u_{\eps} = f\quad\text{in }\Omega,
\end{align}
this equation does not fit into the framework of divergence-form homogenization problems such as \eqref{Intro4}, because of the $\eps^{-1}$ term in front of the first-order term in \eqref{Intro3}.

For the theory of homogenization of nondivergence-form problems such as \eqref{Intro1} we refer to the monograph \cite{BLP11} by Bensoussan, Lions and Papanicolaou, to the paper \cite{AL89} by Avellaneda and Lin, and the references therein. In \cite{BBM86}, Bensoussan, Boccardo and Murat study the more general problem involving a Hamiltonian with quadratic growth. Numerical homogenization for nondivergence-form problems using finite difference schemes has been considered in \cite{FO09} by Froese and Oberman.

The first step in the development of the proposed numerical homogenization scheme is the construction of a finite element method to obtain approximations $(m_h)_{h>0}\subset H^1_{\text{per}}(Y)$ to the invariant measure with optimal order convergence rate
\begin{align*}
\|m-m_h\|_{L^2(Y)}+h\|m-m_h\|_{H^1(Y)}\lesssim h \inf_{\tilde{v}_h\in  \tilde{M}_h} \|m-(\tilde{v}_h+1)\|_{H^1(Y)},
\end{align*}
where $\tilde{M}_h$ denotes the finite-dimensional subspace of $H^1_{\text{per}}(Y)$ consisting of continuous $Y$-periodic piecewise linear functions on the triangulation with zero mean over $Y$; see Theorem 3.1. 

Throughout this work, we use the notation $a\lesssim b$ for $a,b\in \R$ to denote that $a\leq Cb$ for some constant $C>0$ that does not depend on $\eps$ and the discretization parameters.

The second step is to obtain approximations $(A^0_h)_{h>0}\subset \R^{n\times n}$ to the constant matrix $A^0$; see Lemma 3.1. To this end, the integrand in \eqref{Intro26} is replaced by its continuous piecewise linear interpolant and the invariant measure $m$ is replaced by the approximation $m_h$, i.e.,
\begin{align*}
A^0_h:=\int_Y \mathcal{I}_h(Am_h),
\end{align*}
which can be computed exactly using an appropriate quadrature rule. 

The third step is to perform an $H^s(\Omega)$-conforming ($s\in\{1,2\}$) finite element approximation for the problem
\begin{align*}
\left\{\begin{aligned}
A^0_h:D^2 u_0^h &=f\quad\text{in }\Omega,\\
\hfill u_0^h &=0\quad \text{on }\partial\Omega,
\end{aligned}\right.
\end{align*}
on a family of triangulations of the computational domain $\Omega$, parametrized by a discretization parameter $k>0$, measuring the granularity of the triangulation, to obtain $(u_0^{h,k})_{h,k>0}\subset H^s(\Omega)\cap H^1_0(\Omega)$ with 
\begin{align*}
\left\|u_0^h-u_0^{h,k}\right\|_{H^s(\Omega)}\lesssim k\|f\|_{H^{s-1}(\Omega)},
\end{align*}
where the constant is independent of $h$; see Lemma 3.3. Note that for the sake of approximating $u_0$, an $H^1(\Omega)$-conforming finite element method is sufficient.  

The approximation $(u_0^{h,k})_{h,k>0}\subset H^s(\Omega)\cap H^1_0(\Omega)$ obtained by this procedure approximates $u_0$, i.e., the solution to \eqref{Intro25}, with convergence rate
\begin{align*}
\left\|u_0-u_0^{h,k}\right\|_{H^s(\Omega)}\lesssim (h+k)\|f\|_{H^{s-1}(\Omega)},
\end{align*}
which can be improved to $\mathcal{O}(h^2+k)$ for more regular $A$; see Theorem 3.2, Theorem 3.3 and Remark 3.3. 

Concerning the approximation of $u_{\eps}$, i.e., the solution to \eqref{Intro1}, \eqref{Intro2}, we show in Section 2 that under certain assumptions on the domain and the right-hand side, one has that 
\begin{align*}
\left\|u_{\eps}-u_0-\eps^2\sum_{i,j=1}^n \chi_{ij}\left(\frac{\cdot}{\eps}\right)\partial^2_{ij} u_0\right\|_{H^2(\Omega)}\lesssim \sqrt{\eps}\|u_0\|_{W^{2,\infty}(\Omega)} + \eps\|u_0\|_{H^4(\Omega)},
\end{align*}
where the corrector functions $\chi_{ij}:\R^n\rightarrow\R$, $i,j=1,\dots,n$, are defined as the solutions to 
\begin{align*}
\begin{cases}
A:D^2 \chi_{ij} = a_{ij}^0-a_{ij}\quad\text{in }Y,\\
\chi_{ij}\;\text{is $Y$-periodic},\;\int_Y \chi_{ij}=0.
\end{cases}
\end{align*}
This provides us with the estimate 
\begin{align*}
\|u_{\eps}-u_0\|_{H^1(\Omega)}+\sum_{k,l=1}^n \left\|\partial_{kl}^2u_{\eps}-\left(\partial_{kl}^2u_0 + \sum_{i,j=1}^n \left(\partial_{kl}^2\chi_{ij}\right)\left(\frac{\cdot}{\eps}\right)\partial^2_{ij} u_0\right)\right\|_{L^2(\Omega)}=\mathcal{O}(\sqrt{\eps}),
\end{align*}
which shows that $u_0$ is a good $H^1(\Omega)$ approximation to $u_{\eps}$ for small $\eps$, and we show in Sections 3.2 and 3.3 how the above estimate can be used to obtain approximations to $D^2 u_{\eps}$. Note that in order to approximate $u_{\eps}$ in the $H^1(\Omega)$-norm, it is sufficient to approximate $u_0$ in the $H^1(\Omega)$-norm. However, for an approximation of $D^2 u_{\eps}$ based on the above corrector estimate, we need to approximate $u_0$ in the $H^2(\Omega)$-norm.

In Section 3.4, we extend our results to the case of nonuniformly oscillating coefficients, i.e., to problems of the form
\begin{align}\label{water}
\left\{\begin{aligned}
A\left(\cdot,\frac{\cdot}{\eps}\right):D^2 u_{\eps}&=f\quad\text{in }\Omega,\\
\hfill u_{\eps} &= 0\quad\text{on }\partial\Omega,
\end{aligned}\right.
\end{align}
where $A=A(x,y):\Omega\times \R^n\rightarrow \R^{n\times n}$ is a symmetric, uniformly elliptic matrix-valued function that is $Y$-periodic in $y$ for fixed $x\in \Omega$, and such that
\begin{align*}
A\in W^{2,\infty}(\Omega; W^{1,q}(Y) )\text{ for some }q>n.
\end{align*}
We prove the corrector estimate
\begin{align*}
\left\|u_{\eps}-u_0-\eps^2\sum_{i,j=1}^n \chi_{ij}\left(\cdot,\frac{\cdot}{\eps}\right)\partial^2_{ij} u_0\right\|_{H^2(\Omega)}\lesssim \sqrt{\eps}\|u_0\|_{W^{2,\infty}(\Omega)} + \eps\|u_0\|_{H^4(\Omega)},
\end{align*} 
where $u_0$ is the solution to the homogenized problem corresponding to \eqref{water} and $\chi_{ij}$ are certain corrector functions. We then discuss the numerical approximation of $u_{\eps}$ based on this corrector estimate, see Section 3.4.

\section{Homogenization of Elliptic Problems in Nondivergence-Form}

\subsection{Framework}

We denote the unit cell in $\R^n$ by 
\begin{align*}
Y:=(0,1)^n,
\end{align*}
and consider a symmetric matrix-valued function 
\begin{align*}
A=A^{\mathrm{T}}:\R^n\rightarrow \R^{n\times n}
\end{align*}
with the properties
\begin{align}\label{Aass}
\begin{cases}
A\in W^{1,q}(Y)\text{ for some }q\in(n,\infty],\\
A\text{ is } \text{$Y$-periodic},\\
\exists\; \lambda,\Lambda>0:\quad \lambda \lvert \xi\rvert^2\leq  A(y)\xi\cdot \xi\leq \Lambda \lvert\xi\rvert^2\quad\quad\forall\,\xi,y\in\R^n.
\end{cases}
\end{align}
By Sobolev embedding, we then have that
\begin{align*}
A\in C^{0,\alpha}(\R^n)\text{ for some }0<\alpha\leq  1.
\end{align*}
For $\eps>0$, we are concerned with the problem
\begin{align}\label{1}
\left\{\begin{aligned}
A\left(\frac{\cdot}{\eps}\right):D^2 u_{\eps}&=f\quad\text{in }\Omega,\\
\hfill u_{\eps} &= 0\quad\text{on }\partial\Omega,\end{aligned}\right.
\end{align}
where the triple $(\Omega,A,f)$ satisfies one of the following sets of assumptions.
\begin{definition}[Sets of assumptions $\calG^{m,p}$, $\calH^{m}$]
For $m\in \N_0$ and $p\in (1,\infty)$, we define the set of assumptions $\calG^{m,p}$ as
\begin{align*}
(\Omega,A,f)\in \calG^{m,p}\quad \iff\quad \left\{\begin{aligned} \Omega\subset \R^n \text{ is a bounded $C^{2,\gamma}$ domain}, \gamma\in (0,1), \\
A=A^{\mathrm{T}}:\R^n\rightarrow \R^{n\times n} \text{ satisfies \eqref{Aass},}\\
f\in W^{m,p}(\Omega),\end{aligned}\right.  
\end{align*} 
and the set of assumptions $\calH^{m}$ as
\begin{align*}
(\Omega,A,f)\in \calH^{m}\quad \iff\quad \left\{\begin{aligned} \Omega\subset \R^n \text{ is a bounded convex domain,} \\
A=A^{\mathrm{T}}:\R^n\rightarrow \R^{n\times n} \text{ satisfies \eqref{Aass},}\\
\exists\;\delta\in (0,1]: \frac{\lvert A\rvert^2}{(\mathrm{tr}A)^2}\leq \frac{1}{n-1+\delta}\;\text{ in }\R^n,\\
f\in H^{m}(\Omega).\end{aligned}\right.  
\end{align*}
\end{definition}

\begin{remark}
For $n=2$, the Cordes condition, i.e., that there exists a $\delta\in (0,1]$ such that
\begin{align}\label{C4}
\frac{\lvert A(y)\rvert^2}{(\mathrm{tr}A(y))^2}\leq \frac{1}{n-1+\delta}\quad\quad\forall\, y\in\R^n,
\end{align} 
is a consequence of the uniform ellipticity condition. Indeed, for $A=A^{\mathrm{T}}:\R^2\rightarrow \R^{2\times 2}$ satisfying \eqref{Aass}, we have that
\begin{align*}
\frac{\lvert A(y)\rvert^2}{(\mathrm{tr}A(y))^2}=1-\frac{2\det A(y)}{(\mathrm{tr}A(y))^2}\leq 1- \frac{2\lambda^2}{4\Lambda^2}=\frac{1}{1+\delta}\quad\quad\forall\, y\in \R^n
\end{align*}
with $\delta=\frac{\lambda^2}{2\Lambda^2-\lambda^2}\in (0,1]$. Therefore, when $n=2$, the set $\calH^m$ can be simplified to
\begin{align*}
(\Omega,A,f)\in \calH^{m}\quad \iff\quad \left\{\begin{aligned} \Omega\subset \R^n \text{ is a bounded convex domain,} \\
A=A^{\mathrm{T}}:\R^n\rightarrow \R^{n\times n} \text{ satisfies \eqref{Aass},}\\
f\in H^{m}(\Omega).\end{aligned}\right. 
\end{align*}
\end{remark}

The following theorem asserts well-posedness of the problem \eqref{1}; see \cite[Theorem 9.15]{GT01} and \cite[Theorem 3]{SS13}.

\begin{theorem}[Existence and uniqueness of strong solutions]
Assume either that $(\Omega,A,f)\in \calG^{0,p}$ for some $p\in (1,\infty)$, or that $(\Omega,A,f)\in \calH^{0}$ and $p=2$. Then, for any $\eps>0$, the problem \eqref{1} admits a unique solution $u_{\eps}\in W^{2,p}(\Omega)\cap W^{1,p}_0(\Omega)$.
\end{theorem}

\subsection{Transformation into Divergence-Form}

We recall a well-known procedure to transform the problem \eqref{1} into divergence-form; see \cite{AL89,BLP11}. We use the notation
\begin{align*}
W_{\text{per}}(Y):=\left\{ u\in H^1_{\text{per}}(Y):\int_Y u = 0\right\}.
\end{align*}
Let us start by introducing the notion of invariant measure; see \cite{BLP11}.
\begin{lemma}[Invariant measure and solvability condition]
Let $A=A^{\mathrm{T}}:\R^n\rightarrow \R^{n\times n}$ satisfy \eqref{Aass}. Then, there exists a unique solution $m:\R^n\rightarrow \R$ to the problem
\begin{align*}
\begin{cases}
D^2:(Am)=0\quad\text{in }Y,\\
m\;\text{is $Y$-periodic},\; \int_Y m = 1.
\end{cases}
\end{align*}
The function $m$ is called the invariant measure. There holds $m\in W^{1,q}(Y)$, see \cite{BKR01,BS17}, and there exist constants $\bar{m},M>0$ such that 
\begin{align}\label{mbounds}
0<\bar{m}\leq m(y) \leq M\quad\quad\forall\, y\in \R^n.
\end{align}
Moreover, for a $Y$-periodic function $g\in L^2(\R^n)$, the (adjoint) problem \begin{align*}
\begin{cases}
A:D^2u=g\quad\text{in }Y,\\
u\;\text{is $Y$-periodic},\; \int_Y u = 0,
\end{cases}
\end{align*}
admits a solution $u\in W_{\text{per}}(Y)$ if and only if 
\begin{align}\label{solvcond}
\langle g,m\rangle_{L^2(Y)}=0.
\end{align}
\end{lemma}

With the invariant measure at hand, we can easily convert the problem into divergence-form as follows. We define a matrix-valued function $B=(b_{ij})_{1\leq i,j\leq n}:\R^n\rightarrow \R^{n\times n}$ by
\begin{align*}
b_{ij}:= \partial_i v_j - \partial_j v_i, \quad (1\leq i,j\leq n), 
\end{align*}
with $v_l\in W_{\text{per}}(Y)$ denoting the solution to 
\begin{align*}
\begin{cases}
-\Delta v_l=\div(Am)\cdot e_l\quad\text{in }Y,\\
v_l\;\text{is $Y$-periodic},\; \int_Y v_l = 0,
\end{cases}
\end{align*}
for $1\leq l\leq n$. Since $A\in W^{1,q}(Y)$ and $m\in W^{1,q}(Y)$, by elliptic regularity one has that $v_l\in W^{2,q}(Y)$ for any $1\leq l\leq n$. Hence, we have
\begin{align*}
B\in W^{1,q}(Y).
\end{align*}
Further, we observe that $B$ is skew-symmetric, $Y$-periodic with zero mean over $Y$, and that
\begin{align*}
\div(B)=-\div(Am)\quad\text{a.e. on }\R^n.
\end{align*}
Now we let
\begin{align*}
A^{\div}:= Am+B\;\in W^{1,q}(Y).
\end{align*}
Then, since 
\begin{align*}
\div(A^{\div})=0,
\end{align*}
and using the fact that $B$ is skew-symmetric, we obtain
\begin{align*}
\nabla \cdot \left(A^{\div}\left(\frac{\cdot}{\eps}\right)\nabla u_{\eps}\right)=A^{\div}\left(\frac{\cdot}{\eps}\right):D^2 u_{\eps} = (Am)\left(\frac{\cdot}{\eps}\right):D^2 u_{\eps},
\end{align*}
i.e., we have converted \eqref{1} into divergence-form
\begin{align}\label{0,4}
\left\{\begin{aligned}
\nabla \cdot \left(A^{\div}\left(\frac{\cdot}{\eps}\right)\nabla u_{\eps}\right)&=f\;m\left(\frac{\cdot}{\eps}\right)\quad\text{in }\Omega,\\
\hfill u_{\eps} &= 0\hspace{1.8cm}\text{on }\partial\Omega,
\end{aligned}\right.
\end{align}
and it is straightforward to check that $A^{\div}$ is $Y$-periodic, Hölder continuous on $\R^n$ and uniformly elliptic.

\subsection{Uniform $W^{2,p}$ Estimates and Homogenization Theorem}
 
The transformation described in the previous section can be used to obtain uniform $W^{2,p}(\Omega)$ a priori estimates for the solution of \eqref{1}, which are crucial in deriving homogenization results.

\begin{theorem}[Uniform $W^{2,p}$ a priori estimates]
Assume either that $(\Omega,A,f)\in \calG^{0,p}$ for some $p\in (1,\infty)$, or that $(\Omega,A,f)\in \calH^{0}$ and $p=2$. Then, for $\eps\in (0,1]$, the solution $u_{\eps}\in W^{2,p}(\Omega)\cap W^{1,p}_0(\Omega)$ to \eqref{1}, whose existence and uniqueness are guaranteed by Theorem 2.1, satisfies 
\begin{align*}
\|u_{\eps}\|_{W^{2,p}(\Omega)}\lesssim \|f\|_{L^p(\Omega)}
\end{align*}
with the constant absorbed into the notation $\lesssim$ being independent of $\eps$.
\end{theorem}
\begin{proof}
Let us first assume that $(\Omega,A,f)\in \calG^{0,p}$ for some $p\in (1,\infty)$. We showed in the previous section that we can transform problem \eqref{1} into the divergence-form problem \eqref{0,4}, where $A^{\div}:\R^{n}\rightarrow\R^{n\times n}$ is a $Y$-periodic, Hölder continuous, and uniformly elliptic matrix-valued function satisfying
\begin{align*}
\div (A^{\div}) = 0.
\end{align*}
Therefore, we can apply \cite[Theorem D]{AL91} to problem \eqref{0,4} to obtain
\begin{align*}
\|u_{\eps}\|_{W^{2,p}(\Omega)}\lesssim \left\|f\;m\left(\frac{\cdot}{\eps}\right)\right\|_{L^p(\Omega)}\lesssim \|f\|_{L^p(\Omega)}
\end{align*}
with constants independent of $\eps$, where we have used the property \eqref{mbounds} of the invariant measure in the second inequality.

Let us now assume that $(\Omega,A,f)\in \calH^{0}$. Noting that \eqref{C4} implies the Cordes condition for $A\left(\frac{\cdot}{\eps}\right)$ with the same constant $\delta\in (0,1]$ for any $\eps>0$, the proof of \cite[Theorem 3]{SS13} yields the estimate
\begin{align}\label{0,41}
\|u_{\eps}\|_{H^2(\Omega)}\leq \frac{C(n,\mathrm{diam}(\Omega))}{1-\sqrt{1-\delta}}\left\|\gamma\left(\frac{\cdot}{\eps}\right)\right\|_{L^{\infty}(\R^n)}\|f\|_{L^2(\Omega)},
\end{align} 
where $\gamma$ is the function given by
\begin{align*}
\gamma:\R^n\rightarrow\R,\quad \gamma(y):=\frac{\mathrm{tr}A(y)}{\lvert A(y)\rvert^2}.
\end{align*}
We observe that by \eqref{Aass}, there exist constants $\bar{\gamma},\Gamma>0$ such that
\begin{align*}
0<\bar{\gamma}\leq \gamma(y)\leq \Gamma\quad\quad\forall\, y\in \R^n.
\end{align*}
Therefore, we obtain from \eqref{0,41} the bound
\begin{align*}
\|u_{\eps}\|_{H^2(\Omega)}\lesssim \|f\|_{L^2(\Omega)}
\end{align*}
with a constant that is independent of $\eps$.
\end{proof}

This leads to a simple proof of the homogenization theorem for problem \eqref{1}, using the compactness of the embedding $W^{2,p}(\Omega)\hookrightarrow W^{1,p}(\Omega)$ and the fact that we can rewrite the problem as \eqref{0,4}.

\begin{theorem}[Homogenization theorem for nondivergence-form problems]
Assume either that $(\Omega,A,f)\in \calG^{0,p}$ for some $p\in (1,\infty)$, or that $(\Omega,A,f)\in \calH^{0}$ and $p=2$. Then the solution $u_{\eps}\in W^{2,p}(\Omega)\cap W^{1,p}_0(\Omega)$ to \eqref{1} converges weakly in $W^{2,p}(\Omega)$ to the solution $u_0\in W^{2,p}(\Omega)\cap W^{1,p}_0(\Omega)$ of the homogenized problem
\begin{align}\label{1,4}
\left\{
\begin{aligned}
A^0:D^2 u_{0}&=f\quad\text{in }\Omega,\\
\hfill u_{0} &= 0\quad\text{on }\partial\Omega,
\end{aligned}\right.
\end{align}
with $A^0=(a^0_{ij})_{1\leq i,j\leq n}\in \R^{n\times n}$ being the constant matrix whose entries are given by
\begin{align*}
a^0_{ij}:=\int_Y a_{ij}m\quad\quad(1\leq i,j\leq n),
\end{align*}
where $m$ is the invariant measure, as defined in Lemma 2.1.
\end{theorem}
\begin{proof}
By Theorem 2.2, the reflexivity of $W^{2,p}(\Omega)$, the compactness of the embedding $W^{2,p}(\Omega)\hookrightarrow W^{1,p}(\Omega)$, and the properties of the trace operator, there exists a $u_0\in W^{2,p}(\Omega)\cap W^{1,p}_0(\Omega)$ such that (for a subsequence, not indicated,)
\begin{align*}
u_{\eps}&\weak u_0\quad\text{weakly in }W^{2,p}(\Omega),\text{ and}\\
u_{\eps}&\rightarrow u_0\quad\text{strongly in }W^{1,p}(\Omega).
\end{align*}
We can transform \eqref{1} as in Section 2.2 into the divergence-form problem \eqref{0,4} with 
\begin{align*}
A^{\div}=Am+B
\end{align*}
being $Y$-periodic, Hölder continuous and uniformly elliptic on $\R^n$. Recalling that $B$ is of mean zero over $Y$, we have
\begin{align*}
A^{\div}\left(\frac{\cdot}{\eps}\right)\weakstar \int_Y Am=A^0 \quad\text{weakly-$\ast$ in } L^{\infty}(\Omega).
\end{align*}
Since we have that
\begin{align*}
\nabla u_{\eps}\rightarrow \nabla u_0 \quad\text{strongly in } L^p(\Omega),
\end{align*}
we can pass to the limit in the weak formulation of \eqref{0,4} to obtain that $u_0\in W^{2,p}(\Omega)\cap W^{1,p}_0(\Omega)$ solves \eqref{1,4}. We conclude the proof by noting that \eqref{1,4} admits a unique strong solution in $W^{2,p}(\Omega)\cap W^{1,p}_0(\Omega)$.
\end{proof}

\subsection{Correctors}

We show that by adding corrector terms to the solution $u_0$ of the homogenized problem, we obtain a $W^{2,p}$ convergence result.

\begin{theorem}[Corrector estimate I]
Assume either that $(\Omega,A,f)\in \calG^{2,p}$ for some $p\in (1,\infty)$, or that $(\Omega,A,f)\in \calH^{2}$ and $p=2$. Let $\eps\in (0,1]$ and assume that
\begin{align*}
u_0\in W^{4,p}(\Omega).
\end{align*}
Introducing the corrector function $\chi_{ij}$, $1\leq i,j\leq n$, as the solution to 
\begin{align}\label{6,8}
\begin{cases}
A:D^2 \chi_{ij} = a_{ij}^0-a_{ij}\quad\text{in }Y,\\
\chi_{ij}\;\text{is $Y$-periodic},\;\int_Y \chi_{ij}=0,
\end{cases}
\end{align}
and a boundary corrector $\theta_{\eps}$, as the solution to
\begin{align*}
\left\{\begin{aligned}
A\left(\frac{\cdot}{\eps}\right):D^2\theta_{\eps} &= 0\hspace{3.7cm}\text{in }\Omega,\\
\hfill\theta_{\eps} &=- \sum_{i,j=1}^n \chi_{ij}\left(\frac{\cdot}{\eps}\right)\partial^2_{ij} u_0\quad\text{on }\partial\Omega,
\end{aligned}\right.
\end{align*}
the following bound holds: 
\begin{align}\label{7}
\left\|u_{\eps}-u_0-\eps^2\left(\sum_{i,j=1}^n \chi_{ij}\left(\frac{\cdot}{\eps}\right)\partial^2_{ij} u_0+\theta_{\eps}\right)\right\|_{W^{2,p}(\Omega)}\lesssim\eps\|u_0\|_{W^{4,p}(\Omega)}.
\end{align}
\end{theorem}
\begin{proof}
First, we note that since $A\in C^{0,\alpha}(\R^n)$, we have $\chi_{ij}\in C^{2,\alpha}(\R^n)$ for any $1\leq i,j\leq n$ by elliptic regularity theory. A direct computation shows that the function 
\begin{align*}
\tilde{u}_{\eps}:=u_0+\eps^2 \sum_{i,j=1}^n \chi_{ij}\left(\frac{\cdot}{\eps}\right)\partial^2_{ij} u_0
\end{align*} 
solves the problem
\begin{align*}
\left\{\begin{aligned}
A\left(\frac{\cdot}{\eps}\right):D^2\tilde{u}_{\eps} &= f+\eps F_{\eps}\hspace{2.4cm}\text{ in }\Omega,\\
\hfill \tilde{u}_{\eps} &=\eps^2 \sum_{i,j=1}^n \chi_{ij}\left(\frac{\cdot}{\eps}\right)\partial^2_{ij} u_0\quad\text{on }\partial\Omega,
\end{aligned}\right.
\end{align*}
where
\begin{align*}
F_{\eps}:=\sum_{i,j,k,l=1}^n a_{ij}\left(\frac{\cdot}{\eps}\right)\left(2\partial_{i} \chi_{kl}\left(\frac{\cdot}{\eps}\right) \partial^3_{jkl}u_0+\eps \chi_{kl}\left(\frac{\cdot}{\eps}\right)\partial^4_{ijkl}u_0\right).
\end{align*}
Note that since $u_0\in W^{4,p}(\Omega)$, one has that
\begin{align*}
\|F_{\eps}\|_{L^p(\Omega)}\lesssim \|u_0\|_{W^{4,p}(\Omega)},
\end{align*}
with the constant being independent of $\eps$. We then have that $d_{\eps}:=\tilde{u}_{\eps}-u_{\eps}$ satisfies
\begin{align*}
\left\{\begin{aligned}
A\left(\frac{\cdot}{\eps}\right):D^2 d_{\eps} &= \eps F_{\eps}\hspace{3.3cm}\text{in }\Omega,\\
\hfill d_{\eps} &=\eps^2 \sum_{i,j=1}^n \chi_{ij}\left(\frac{\cdot}{\eps}\right)\partial^2_{ij} u_0\quad\text{on }\partial\Omega.
\end{aligned}\right.
\end{align*}
Therefore, by the definition of the boundary corrector,
\begin{align*}
\left\{\begin{aligned}
A\left(\frac{\cdot}{\eps}\right):D^2\left(d_{\eps}+\eps^2\theta_{\eps}\right) &= \eps F_{\eps}\quad\;\text{in }\Omega,\\
\hfill d_{\eps}+\eps^2\theta_{\eps} &=0\quad\;\;\;\text{ on }\partial\Omega.
\end{aligned}\right.
\end{align*}
We conclude using the estimate from Theorem 2.2 that
\begin{align*}
\|d_{\eps}+\eps^2\theta_{\eps}\|_{W^{2,p}(\Omega)}\lesssim \eps\|F_{\eps}\|_{L^p(\Omega)} \lesssim  \eps\|u_0\|_{W^{4,p}(\Omega)},
\end{align*}
and \eqref{7} holds.
\end{proof}

The following theorem shows that if $u_0\in W^{4,p}(\Omega)\cap W^{2,\infty}(\Omega)$, then we can absorb the term involving the boundary corrector into the right-hand side at the cost of powers of $\eps$.

\begin{theorem}[Corrector estimate II]
Assume either that $(\Omega,A,f)\in \calG^{2,p}$ for some $p\in (1,\infty)$, or that $(\Omega,A,f)\in \calH^{2}$ and $p=2$. Let $\eps\in (0,1]$ and assume that
\begin{align}\label{addregu0}
u_0\in W^{4,p}(\Omega)\cap W^{2,\infty}(\Omega).
\end{align} 
Then, 
\begin{align*}
\left\|u_{\eps}-u_0-\eps^2\sum_{i,j=1}^n \chi_{ij}\left(\frac{\cdot}{\eps}\right)\partial^2_{ij} u_0\right\|_{W^{2,p}(\Omega)}\lesssim \eps^{\frac{1}{p}}\|u_0\|_{W^{2,\infty}(\Omega)} + \eps\|u_0\|_{W^{4,p}(\Omega)}.
\end{align*}
\end{theorem}
\begin{proof}
Let $\eta\in C_c^{\infty}(\R^n)$ be a cut-off function with $0\leq \eta\leq 1$,
\begin{align*}
\eta &\equiv 1\quad\text{in }\left\{x\in\Omega: \mathrm{dist}(x,\partial\Omega)<\frac{\eps}{2}\right\},\\
\eta &\equiv 0\quad\text{in }\left\{x\in\Omega: \mathrm{dist}(x,\partial\Omega)\geq\eps\right\},
\end{align*}
and let $\eta$ satisfy 
\begin{align*}
\lvert \nabla \eta\rvert+\eps\lvert D^2 \eta\rvert\lesssim \frac{1}{\eps}\quad\text{in }\Omega.
\end{align*}
We introduce the function
\begin{align*}
\tilde{\theta}_{\eps}:= \theta_{\eps}+\eta \sum_{i,j=1}^n \chi_{ij}\left(\frac{\cdot}{\eps}\right)\partial_{ij}^2u_0,
\end{align*}
and verify that 
\begin{align*}
A&\left(\frac{\cdot}{\eps}\right):D^2\tilde{\theta}_{\eps}=\sum_{i,j,k,l=1}^n a_{ij}\left(\frac{\cdot}{\eps}\right)\partial^2_{ij}\left(\eta\; \chi_{kl}\left(\frac{\cdot}{\eps}\right)\partial_{kl}^2u_0\right)=\frac{1}{\eps^2}S_1 + \frac{1}{\eps}S_2 + S_3,
\end{align*}
where $S_1,S_2$ and $S_3$ are given by
\begin{align*}
S_1&:=\sum_{i,j,k,l=1}^n a_{ij}\left(\frac{\cdot}{\eps}\right)\eta\;\partial^2_{ij}\chi_{kl}\left(\frac{\cdot}{\eps}\right)\partial^2_{kl}u_0,\\
S_2&:=2\sum_{i,j,k,l=1}^n a_{ij}\left(\frac{\cdot}{\eps}\right)\left(\partial_i \eta\; \partial_j \chi_{kl}\left(\frac{\cdot}{\eps}\right)\partial^2_{kl}u_0+\eta\; \partial_i\chi_{kl}\left(\frac{\cdot}{\eps}\right)\partial^3_{jkl}u_0\right),\\
S_3&:=\sum_{i,j,k,l=1}^n a_{ij}\left(\frac{\cdot}{\eps}\right)\left(\partial^2_{ij}\eta\; \chi_{kl}\left(\frac{\cdot}{\eps}\right)\partial^2_{kl}u_0+2\partial_i \eta\;\chi_{kl}\left(\frac{\cdot}{\eps}\right) \partial^3_{jkl}u_0 +\eta\; \chi_{kl}\left(\frac{\cdot}{\eps}\right)\partial^4_{ijkl}u_0\right).
\end{align*}
Therefore, $\tilde{\theta}_{\eps}$ satisfies 
\begin{align*}
\left\{\begin{aligned}
A\left(\frac{\cdot}{\eps}\right):D^2\tilde{\theta}_{\eps} &= \frac{1}{\eps^2}S_1 + \frac{1}{\eps}S_2 + S_3\quad\text{ in }\Omega,\\
\hspace{1.9cm}\tilde{\theta}_{\eps} &= 0\hspace{3.4cm}\text{on }\partial\Omega.
\end{aligned}\right.
\end{align*}
Since $u_0\in W^{4,p}(\Omega)\cap W^{2,\infty}(\Omega)$ by assumption, the right-hand side belongs to $L^p(\Omega)$, and we have by Theorem 2.2 that
\begin{align*}
\left\|\tilde{\theta}_{\eps}\right\|_{W^{2,p}(\Omega)}\lesssim \frac{1}{\eps^2}\|S_1\|_{L^p(\Omega)} + \frac{1}{\eps}\|S_2\|_{L^p(\Omega)} + \|S_3\|_{L^p(\Omega)}.
\end{align*}
We look at the terms on the right-hand side separately and start with $S_1$. Using the boundedness of $A$ and the fact that $\chi_{ij}\in W^{2,\infty}(\R^n)$, we have
\begin{align*}
\|S_1\|_{L^p(\Omega)}&=\left\|\sum_{i,j,k,l=1}^n a_{ij}\left(\frac{\cdot}{\eps}\right)\eta\;\partial^2_{ij}\chi_{kl}\left(\frac{\cdot}{\eps}\right)\partial^2_{kl}u_0  \right\|_{L^p(\Omega)}\\
&\lesssim  \|u_0\|_{W^{2,\infty}(\Omega)}  \|\eta\|_{L^p(\Omega)}\\
&\lesssim \lvert\left\{x\in\Omega: \mathrm{dist}(x,\partial\Omega)<\eps\right\}   \rvert^{\frac{1}{p}}\;\|u_0\|_{W^{2,\infty}(\Omega)}\\ &\lesssim \eps^{\frac{1}{p}}\|u_0\|_{W^{2,\infty}(\Omega)}.
\end{align*}
For $S_2$, we obtain similarly that
\begin{align*}
\|S_2\|_{L^p(\Omega)}&=\left\|2\sum_{i,j,k,l=1}^n a_{ij}\left(\frac{\cdot}{\eps}\right)\left(\partial_i \eta\; \partial_j \chi_{kl}\left(\frac{\cdot}{\eps}\right)\partial^2_{kl}u_0+\eta\; \partial_i\chi_{kl}\left(\frac{\cdot}{\eps}\right)\partial^3_{jkl}u_0\right)\right\|_{L^p(\Omega)}\\ 
&\lesssim \|\nabla \eta\|_{L^p(\Omega)}\|u_0\|_{W^{2,\infty}(\Omega)}+\| \eta\|_{L^{\infty}(\Omega)}\|u_0\|_{W^{4,p}(\Omega)}  \\
&\lesssim \frac{1}{\eps}\;\lvert\left\{x\in\Omega: \mathrm{dist}(x,\partial\Omega)<\eps\right\}   \rvert^{\frac{1}{p}} \;\|u_0\|_{W^{2,\infty}(\Omega)} +\|u_0\|_{W^{4,p}(\Omega)}\\
&\lesssim \frac{1}{\eps^{1-\frac{1}{p}}} \|u_0\|_{W^{2,\infty}(\Omega)} +\|u_0\|_{W^{4,p}(\Omega)}. 
\end{align*}
Finally, for $S_3$, we have that
\begin{align*}
\|S_3\|_{L^p(\Omega)}&=\left\|\sum_{i,j,k,l=1}^n a_{ij}\left(\frac{\cdot}{\eps}\right)\left(\partial^2_{ij}\eta\; \chi_{kl}\left(\frac{\cdot}{\eps}\right)\partial^2_{kl}u_0  \right.\right. \\ &\hspace{3cm}\left.\left.+2\partial_i \eta\;\chi_{kl}\left(\frac{\cdot}{\eps}\right) \partial^3_{jkl}u_0 +\eta\; \chi_{kl}\left(\frac{\cdot}{\eps}\right)\partial^4_{ijkl}u_0\right)\right\|_{L^p(\Omega)}\\
&\lesssim \|D^2 \eta\|_{L^p(\Omega)}\|u_0\|_{W^{2,\infty}(\Omega)}+\left(\|\nabla\eta\|_{L^{\infty}(\Omega)} + \|\eta\|_{L^{\infty}(\Omega)} \right)\|u_0\|_{W^{4,p}(\Omega)}\\
&\lesssim \frac{1}{\eps^2}\;\lvert\left\{x\in\Omega: \mathrm{dist}(x,\partial\Omega)<\eps\right\}   \rvert^{\frac{1}{p}} \;\|u_0\|_{W^{2,\infty}(\Omega)}+\frac{1}{\eps} \|u_0\|_{W^{4,p}(\Omega)}\\
&\lesssim \frac{1}{\eps^{2-\frac{1}{p}}} \|u_0\|_{W^{2,\infty}(\Omega)} + \frac{1}{\eps} \|u_0\|_{W^{4,p}(\Omega)}.
\end{align*}
Altogether, we have shown that
\begin{align*}
\left\|\tilde{\theta}_{\eps}\right\|_{W^{2,p}(\Omega)}&\lesssim \left(\frac{ \eps^{\frac{1}{p}}}{\eps^2}+ \frac{1}{\eps\cdot \eps^{1-\frac{1}{p}}}+\frac{1}{\eps^{2-\frac{1}{p}}}\right)\|u_0\|_{W^{2,\infty}(\Omega)}+\left(\frac{1}{\eps}+\frac{1}{\eps} \right)\|u_0\|_{W^{4,p}(\Omega)}\\
&\lesssim \frac{1}{\eps^{2-\frac{1}{p}}}\|u_0\|_{W^{2,\infty}(\Omega)}+\frac{1}{\eps}  \|u_0\|_{W^{4,p}(\Omega)}. 
\end{align*}
By direct computation, using the bounds 
\begin{align*}
\|\eta\|_{L^p(\Omega)}\lesssim \eps^{\frac{1}{p}},\quad \|\nabla \eta\|_{L^p(\Omega)}\lesssim \frac{1}{\eps^{1-\frac{1}{p}}},\quad \|D^2 \eta\|_{L^p(\Omega)}\lesssim \frac{1}{\eps^{2-\frac{1}{p}}},
\end{align*}
we can show that
\begin{align*}
\left\|\eta \sum_{i,j=1}^n \chi_{ij}\left(\frac{\cdot}{\eps}\right)\partial_{ij}^2u_0  \right\|_{W^{2,p}(\Omega)}\lesssim \frac{1}{\eps^{2-\frac{1}{p}}}\|u_0\|_{W^{2,\infty}(\Omega)}+\frac{1}{\eps}  \|u_0\|_{W^{4,p}(\Omega)}.
\end{align*}
Therefore, using the triangle inequality, we obtain that
\begin{align*}
\|\theta_{\eps}\|_{W^{2,p}(\Omega)}&\lesssim \frac{1}{\eps^{2-\frac{1}{p}}}\|u_0\|_{W^{2,\infty}(\Omega)}+\frac{1}{\eps} \|u_0\|_{W^{4,p}(\Omega)}.
\end{align*}
We conclude that
\begin{align*}
\|\eps^2\theta_{\eps}\|_{W^{2,p}(\Omega)}&\lesssim \eps^{\frac{1}{p}}\|u_0\|_{W^{2,\infty}(\Omega)} + \eps\|u_0\|_{W^{4,p}(\Omega)}. 
\end{align*}
The claim now follows from \eqref{7}.
\end{proof}

Let us remark that $W^{4,p}(\Omega)\hookrightarrow W^{2,\infty}(\Omega)$ for $p>\frac{n}{2}$, i.e., assumption \eqref{addregu0} is a consequence of $u_0\in W^{4,p}(\Omega)$; in particular, for dimensions $n\in \{2,3\}$ and $p=2$, one can replace condition \eqref{addregu0} by $u_0\in H^4(\Omega)$.

Let us recall that $u_0$ is the solution to the elliptic constant-coefficient problem \eqref{1,4}. For bounded convex polygonal domains ($n=2$), $u_0\in H^4(\Omega)$ can be ensured by assuming that $f\in H^2(\Omega)$ satisfies certain compatibility conditions at the corners of the domain. In the case of Poisson's equation on $\Omega=(0,1)^2$, a necessary and sufficient condition for $u_0\in H^4(\Omega)\cap H^1_0(\Omega)$ is that $f\in H^2(\Omega)$ and $f=0$ at the corners of $\Omega$, see \cite{HO14}. We note that these conditions are satisfied for functions $f\in H^2(\Omega)$ such that $\supp(f)\Subset \Omega$, see \cite{Gri11}.

\section{The Numerical Scheme}

\subsection{Numerical Homogenization Scheme}
The first step is to approximate the invariant measure.
\subsubsection{Approximation of $m$}

For the approximation of the invariant measure $m$, we consider a shape-regular triangulation of $Y$ into triangles with longest edge $h>0$ and let 
\begin{align*}
\tilde{M}_h\subset W_{\text{per}}(Y)=\left\{v\in H^1_{per}(Y):\int_Y v = 0\right\}
\end{align*}
be the finite-dimensional subspace of $W_{\text{per}}(Y)$ consisting of continuous $Y$-periodic piecewise linear functions on the triangulation with zero mean over $Y$. We assume that  
\begin{align*}
W_{\text{per}}(Y)= \overline{\bigcup_{h>0} \tilde{M}_h}.
\end{align*}
Then we have the following approximation result for $m$.
\begin{theorem}[Approximation of the invariant measure]
Let $A=A^{\mathrm{T}}:\R^n\rightarrow \R^{n\times n}$ satisfy \eqref{Aass}. Then, for $h>0$ sufficiently small, there exists a unique $\tilde{m}_h\in\tilde{M}_h$ such that
\begin{align}\label{8}
\int_Y \left(A\nabla \tilde{m}_h+\tilde{m}_h\;\div A\right)\cdot \nabla \fhi_h=-\int_Y (\div A)\cdot \nabla \fhi_h\quad\quad\forall\,\fhi_h\in\tilde{M}_h,
\end{align}
and writing 
\begin{align*}
m_h:=\tilde{m}_h+1,
\end{align*}
we have that
\begin{align*}
\|m-m_h\|_{L^2(Y)}+h\|m-m_h\|_{H^1(Y)}\lesssim h \inf_{\tilde{v}_h\in  \tilde{M}_h} \|m-(\tilde{v}_h+1)\|_{H^1(Y)},
\end{align*}
where $m$ is the invariant measure, as defined in Lemma 2.1.
\end{theorem}
\begin{remark}
In particular, since 
\begin{align*}
\inf_{\tilde{v}_h\in  \tilde{M}_h} \|m-(\tilde{v}_h+1)\|_{H^1(Y)} = o(1),
\end{align*}
we have that
\begin{align*}
m_h\rightarrow m\quad\text{in }H^1(Y)
\end{align*}
as $h$ tends to zero.
\end{remark}
\begin{proof}[Proof of Theorem 3.1]
We observe that $m=\tilde{m}+1$, where $\tilde{m}$ is the unique solution to the problem
\begin{align*}
\begin{cases}
-\nabla\cdot \left(A\nabla \tilde{m}+\tilde{m}\;\div A\right)=\nabla\cdot (\div A)\quad\text{in }Y,\\
\tilde{m}\;\text{is $Y$-periodic},\; \int_Y \tilde{m} = 0,
\end{cases}
\end{align*}
i.e.,
\begin{align*}
\tilde{m}\in W_{\text{per}}(Y),\quad a(\tilde{m},\fhi)=-\int_Y (\div A)\cdot \nabla \fhi\quad\quad\forall\,\fhi\in W_{\text{per}}(Y),
\end{align*}
where
\begin{align*}
a:W_{\text{per}}(Y)\times W_{\text{per}}(Y)\longrightarrow\R,\quad a(u,v):=\int_Y A\nabla u\cdot \nabla v +\int_Y u(\div A)\cdot \nabla v.
\end{align*}
We further observe that \eqref{8} is equivalent to
\begin{align}\label{8,5}
\tilde{m}_h\in \tilde{M}_h,\quad a(\tilde{m}_h,\fhi_h)=-\int_Y (\div A)\cdot \nabla \fhi_h\quad\quad\forall\,\fhi_h\in \tilde{M}_h.
\end{align}
We start by showing boundedness of $a$ and a G{\aa}rding-type inequality. We claim that there exist constants $C_b,C_g>0$ such that
\begin{align}\label{9}
\lvert a(u,v)\rvert\leq C_b \|u\|_{H^1(Y)}\|v\|_{H^1(Y)}\quad\quad\forall\, u,v\in W_{\text{per}}(Y),
\end{align}
and
\begin{align}\label{10}
a(u,u)\geq \frac{\lambda}{2} \|u\|_{H^1(Y)}^2-C_g \|u\|_{L^2(Y)}^2\quad\quad\forall\, u\in W_{\text{per}}(Y).
\end{align}

Let us first show \eqref{9}. For $u,v\in W_{\text{per}}(Y)$, by Hölder's inequality and Sobolev embeddings (note that, according to \eqref{Aass}, $q>n$), we have that
\begin{align*}
\left\lvert \int_Y u(\div A)\cdot \nabla v \right\rvert \leq \|\div A\|_{L^q(Y)}\|u\|_{L^{\frac{2q}{q-2}}(Y)}\|\nabla v\|_{L^2(Y)}\lesssim \|u\|_{H^1(Y)}\|v\|_{H^1(Y)}.
\end{align*}
Using the fact that $A\in W^{1,q}(Y)\hookrightarrow L^{\infty}(Y)$ since $q>n$, we obtain the bound
\begin{align*}
\lvert a(u,v)\rvert\leq \left\lvert \int_Y A\nabla u\cdot \nabla v \right\rvert+   \left\lvert \int_Y u(\div A)\cdot \nabla v \right\rvert\lesssim \|u\|_{H^1(Y)}\|v\|_{H^1(Y)}
\end{align*}
for any $u,v\in W_{\text{per}}(Y)$, i.e. \eqref{9} holds.

Let us now show the estimate \eqref{10}. For $u\in W_{\text{per}}(Y)$, by ellipticity and Hölder's inequality, we have
\begin{align*}
a(u,u) &=\int_Y A\nabla u\cdot \nabla u +\int_Y u(\div A)\cdot \nabla u\\
&\geq \lambda \|\nabla u\|_{L^2(Y)}^2-\|\div A\|_{L^q(Y)}\|u\|_{L^{\frac{2q}{q-2}}(Y)}\|\nabla u\|_{L^2(Y)}.
\end{align*}
For the second term we use the Gagliardo--Nirenberg inequality and Young's inequality to obtain
\begin{align*}
\|\div A\|_{L^q(Y)}\|u\|_{L^{\frac{2q}{q-2}}(Y)}\|\nabla u\|_{L^2(Y)}&\leq C(q,n)\|\div A\|_{L^q(Y)} \|u\|_{L^{2}(Y)}^{1-\frac{n}{q}}\|\nabla u\|_{L^2(Y)}^{1+\frac{n}{q}}\\
&\leq \frac{\lambda}{2}\|\nabla u\|_{L^2(Y)}^{2}+C(q,n,\lambda,\|\div A\|_{L^q(Y)})\|u\|_{L^2(Y)}^{2}.
\end{align*}
Therefore, we have
\begin{align*}
a(u,u)&\geq  \frac{\lambda}{2}\|\nabla u\|_{L^2(Y)}^{2} - C(q,n,\lambda,\|\div A\|_{L^q(Y)})\|u\|_{L^2(Y)}^{2}\\
&=\frac{\lambda}{2}\| u\|_{H^1(Y)}^{2}-\left(\frac{\lambda}{2}+C(q,n,\lambda,\|\div A\|_{L^q(Y)})  \right)\|u\|_{L^2(Y)}^{2}
\end{align*}
for any $u\in W_{\text{per}}(Y)$, i.e., \eqref{10} holds with 
\begin{align*}
C_g:=\frac{\lambda}{2}+C(q,n,\lambda,\|\div A\|_{L^q(Y)}).
\end{align*}
We use Schatz's method to derive an a priori estimate; see \cite{Sch74}.

From our G{\aa}rding-type inequality \eqref{10} we see that (note $\tilde{m}-\tilde{m}_h\in W_{\text{per}}(Y)$)
\begin{align}\label{longestimate0}
\begin{split}
\|\tilde{m}-\tilde{m}_h\|_{H^1(Y)}-\frac{2C_g}{\lambda}\|\tilde{m}-\tilde{m}_h\|_{L^2(Y)}&\leq \|\tilde{m}-\tilde{m}_h\|_{H^1(Y)}-\frac{2C_g}{\lambda}\frac{\|\tilde{m}-\tilde{m}_h\|_{L^2(Y)}^2}{\|\tilde{m}-\tilde{m}_h\|_{H^1(Y)}}\\
&\leq \frac{2}{\lambda}\frac{a(\tilde{m}-\tilde{m}_h,\tilde{m}-\tilde{m}_h)}{\|\tilde{m}-\tilde{m}_h\|_{H^1(Y)}}.
\end{split}
\end{align}
By Galerkin-orthogonality and boundedness, we have for any $\tilde{v}_h\in \tilde{M}_h$ that
\begin{align*}
\frac{a(\tilde{m}-\tilde{m}_h,\tilde{m}-\tilde{m}_h)}{\|\tilde{m}-\tilde{m}_h\|_{H^1(Y)}} = \frac{a(\tilde{m}-\tilde{m}_h,\tilde{m}-\tilde{v}_h)}{\|\tilde{m}-\tilde{m}_h\|_{H^1(Y)}} \leq C_b \|\tilde{m}-\tilde{v}_h\|_{H^1(Y)},
\end{align*}
and taking the infimum over all $\tilde{v}_h\in \tilde{M}_h$, we find
\begin{align*}
\frac{a(\tilde{m}-\tilde{m}_h,\tilde{m}-\tilde{m}_h)}{\|\tilde{m}-\tilde{m}_h\|_{H^1(Y)}}\leq C_b \inf_{\tilde{v}_h\in  \tilde{M}_h} \|\tilde{m}-\tilde{v}_h\|_{H^1(Y)}.
\end{align*}
Combining this estimate with \eqref{longestimate0} yields
\begin{align}\label{longestimate}
\|\tilde{m}-\tilde{m}_h\|_{H^1(Y)}-\frac{2C_g}{\lambda}\|\tilde{m}-\tilde{m}_h\|_{L^2(Y)}\leq \frac{2C_b}{\lambda}\inf_{v_h\in  \tilde{M}_h} \|\tilde{m}-\tilde{v}_h\|_{H^1(Y)}.
\end{align}
Next, we use an Aubin--Nitsche-type duality argument.

Let $\phi\in W_{\text{per}}(Y)$ be the unique solution to 
\begin{align}\label{11}
\begin{split}
\begin{cases}
-\nabla\cdot\left(A\nabla\phi\right)+(\div A)\cdot \nabla \phi =\frac{\tilde{m}-\tilde{m}_h}{m}\quad\text{in }Y,\\
\phi\; \text{is $Y$-periodic},\;\int_Y \phi =0.
\end{cases}
\end{split}
\end{align} 
We note that the solvability condition \eqref{solvcond} is satisfied:
\begin{align*}
\int_Y \frac{\tilde{m}-\tilde{m}_h}{m}m=\int_Y (\tilde{m}-\tilde{m}_h)=0.
\end{align*}
We have, using the bounds on the invariant measure \eqref{mbounds}, the weak formulation of \eqref{11} and the symmetry of $A$, that
\begin{align*}
\frac{1}{M}\|\tilde{m}-\tilde{m}_h\|_{L^2(Y)}^2&\leq \int_Y \frac{\tilde{m}-\tilde{m}_h}{m}(\tilde{m}-\tilde{m}_h)\\
&= \int_Y A\nabla \phi\cdot \nabla(\tilde{m}-\tilde{m}_h)+\int_Y (\div A)\cdot \nabla \phi \;(\tilde{m}-\tilde{m}_h)\\
&=\int_Y A\nabla(\tilde{m}-\tilde{m}_h)\cdot \nabla\phi+\int_Y (\tilde{m}-\tilde{m}_h)(\div A)\cdot \nabla \phi.
\end{align*}
Next, we use Galerkin orthogonality, the boundedness \eqref{9} and an interpolation inequality to obtain
\begin{align*}
\frac{1}{M}\|\tilde{m}-\tilde{m}_h\|_{L^2(Y)}^2&\leq a(\tilde{m}-\tilde{m}_h,\phi)\\
&= a(\tilde{m}-\tilde{m}_h,\phi-\mathcal{I}_h\phi)\\
&\lesssim \|\tilde{m}-\tilde{m}_h\|_{H^1(Y)} \|\phi-\mathcal{I}_h\phi\|_{H^1(Y)}\\
&\lesssim h\|\tilde{m}-\tilde{m}_h\|_{H^1(Y)}\|\phi\|_{H^2(Y)},
\end{align*} 
where $\mathcal{I}_h\phi$ denotes the continuous piecewise linear interpolant of $\phi$ on the triangulation.
Finally, by a regularity estimate for $\phi$ and the bounds on the invariant measure \eqref{mbounds}, we arrive at the bound
\begin{align*}
\|\phi\|_{H^2(Y)}\lesssim \left\|\frac{\tilde{m}-\tilde{m}_h}{m}\right\|_{L^2(Y)}\lesssim \left\|\tilde{m}-\tilde{m}_h\right\|_{L^2(Y)},
\end{align*}
which provides us with the estimate
\begin{align*}
\|\tilde{m}-\tilde{m}_h\|_{L^2(Y)}\leq C_0 h \|\tilde{m}-\tilde{m}_h\|_{H^1(Y)}
\end{align*}
for some $C_0>0$. Combining this with \eqref{longestimate} we have
\begin{align*}
\left(1-\frac{2C_gC_0}{\lambda}h\right)\|\tilde{m}-\tilde{m}_h\|_{H^1(Y)}&\leq \|\tilde{m}-\tilde{m}_h\|_{H^1(Y)}-\frac{2C_g}{\lambda}\|\tilde{m}-\tilde{m}_h\|_{L^2(Y)}\\
&\leq \frac{2C_b}{\lambda}\inf_{\tilde{v}_h\in  \tilde{M}_h} \|\tilde{m}-\tilde{v}_h\|_{H^1(Y)}.
\end{align*}
Therefore, for $h$ sufficiently small, we arrive at the bounds
\begin{align*}
\|\tilde{m}-\tilde{m}_h\|_{H^1(Y)}\lesssim \inf_{\tilde{v}_h\in  \tilde{M}_h} \|\tilde{m}-\tilde{v}_h\|_{H^1(Y)},
\end{align*}
and
\begin{align*}
\|\tilde{m}-\tilde{m}_h\|_{L^2(Y)}\leq C_0 h \|\tilde{m}-\tilde{m}_h\|_{H^1(Y)}\lesssim h\inf_{\tilde{v}_h\in  \tilde{M}_h} \|\tilde{m}-\tilde{v}_h\|_{H^1(Y)}.
\end{align*}
We have thus established the a priori estimate
\begin{align*}
\|\tilde{m}-\tilde{m}_h\|_{L^2(Y)} + h \|\tilde{m}-\tilde{m}_h\|_{H^1(Y)}\lesssim h \inf_{\tilde{v}_h\in  \tilde{M}_h} \|\tilde{m}-\tilde{v}_h\|_{H^1(Y)},
\end{align*}
which immediately implies existence and uniqueness of solutions to \eqref{8,5}.

Finally, using that $m=\tilde{m}+1$ and $m_h=\tilde{m}_h+1$, we conclude that
\begin{align*}
\|m-m_h\|_{L^2(Y)}+h\|m-m_h\|_{H^1(Y)}\lesssim h \inf_{\tilde{v}_h\in  \tilde{M}_h} \|m-(\tilde{v}_h+1)\|_{H^1(Y)}.
\end{align*}
\end{proof}

\subsubsection{Approximation of $A^0$}

We use this finite element approximation of the invariant measure to obtain an approximation to the constant matrix
\begin{align*}
A^0=\int_Y Am.
\end{align*}
To this end, we first replace the invariant measure $m$ by the approximation $m_h$ from Theorem 3.1, and then replace the integrand by its piecewise linear interpolant,
\begin{align*}
A^0_h:=\int_Y \mathcal{I}_h(Am_h).
\end{align*}
This integral can be computed exactly using an appropriate quadrature rule. The following lemma gives an error estimate for this approximation.

\begin{lemma}[Approximation of $A^0$]
Let $A=A^{\mathrm{T}}:\R^n\rightarrow \R^{n\times n}$ satisfy \eqref{Aass}. Further, let $A^0=(a_{ij}^0)\in \R^{n\times n}$ be the constant matrix given by Theorem 2.3, let $m_h$ be the approximation to the invariant measure given by Theorem 3.1, and let $A_h^0=(a_{ij,h}^0)\in \R^{n\times n}$ be the matrix given by
\begin{align*}
a_{ij,h}^0:=\int_Y \mathcal{I}_h(a_{ij}m_h),\quad\quad 1\leq i,j\leq n.
\end{align*}
Then, for $h>0$ sufficiently small, $A_h^0$ is elliptic and 
\begin{align*}
\max_{1\leq i,j\leq n}\;\left\lvert a_{ij}^0-a_{ij,h}^0\right\rvert \lesssim h .
\end{align*}
\end{lemma}
\begin{proof}
Fix $1\leq i,j\leq n$. Using the definition of $A^0=(a_{ij}^0)$, i.e.,
\begin{align*}
a_{ij}^0=\int_Y a_{ij}m,
\end{align*}
we obtain the estimate
\begin{align*}
\lvert a_{ij}^0-a_{ij,h}^0\rvert \leq \|a_{ij}(m-m_h)\|_{L^1(Y)}+\|a_{ij}m_h-\mathcal{I}_h(a_{ij}m_h)\|_{L^1(Y)}.
\end{align*}
For the first term, we have
\begin{align*}
\|a_{ij}(m-m_h)\|_{L^1(Y)}\lesssim \|m-m_h\|_{L^1(Y)}\lesssim \|m-m_h\|_{L^2(Y)}.
\end{align*}
For the second term, let us first note that using $a_{ij}\in W^{1,q}(Y)$ with $q>n$ and Sobolev embeddings, we have
\begin{align*}
\lvert a_{ij}m_h\rvert_{H^1(Y)} &\leq \|\nabla a_{ij}\|_{L^q(Y)}\|m_h\|_{L^{\frac{2q}{q-2}}(Y)}+\|a_{ij}\|_{L^{\infty}(Y)}\|\nabla m_h\|_{L^2(Y)}\\
&\lesssim \|a_{ij}\|_{W^{1,q}(Y)}\|m_h\|_{H^1(Y)}.
\end{align*}
Therefore, using a standard interpolation error bound, we obtain
\begin{align*}
\|a_{ij}m_h-\mathcal{I}_h(a_{ij}m_h)\|_{L^1(Y)} &\lesssim \|a_{ij}m_h-\mathcal{I}_h(a_{ij}m_h)\|_{L^2(Y)} \\
&\lesssim h \lvert a_{ij}m_h\rvert_{H^1(Y)}\\
&\lesssim h\|a_{ij}\|_{W^{1,q}(Y)}\|m_h\|_{H^1(Y)}.
\end{align*}
By Theorem 3.1, for $h>0$ sufficiently small, we have that
\begin{align*}
\lvert a_{ij}^0-a_{ij,h}^0\rvert &\lesssim \|m-m_h\|_{L^2(Y)}+h\|m_h\|_{H^1(Y)}\\
&\lesssim \|m-m_h\|_{L^2(Y)}+h\|m-m_h\|_{H^1(Y)}+h\|m\|_{H^1(Y)}\\
&\lesssim h \inf_{\tilde{v}_h\in  \tilde{M}_h} \|m-(\tilde{v}_h+1)\|_{H^1(Y)} + h\|m\|_{H^1(Y)}\\
&\lesssim h\|m-1\|_{H^1(Y)}+ h\|m\|_{H^1(Y)}\\
&\lesssim h.
\end{align*}
Finally, we note that this implies that for $h>0$ sufficiently small, $A_h^0$ is elliptic. 
\end{proof}

\subsubsection{Approximation of $u_0$}

For the approximation of the solution $u_0$ to the homogenized problem, we use the following comparison result for the error committed when replacing $A^0$ by $A^0_h$.

\begin{lemma}[Comparison result]
Assume either that $(\Omega,A,f)\in \calG^{0,2}$ or that $(\Omega,A,f)\in \calH^{0}$. Let $A_h^0\in \R^{n\times n}$ be the approximation to $A^0$ as in Lemma 3.1.  Then, for $h>0$ sufficiently small, we have that
\begin{align*}
\|u_0-u_0^h\|_{H^2(\Omega)}\lesssim h \|f\|_{L^2(\Omega)},
\end{align*}
where $u_0^h\in H^2(\Omega)\cap H^1_0(\Omega)$ is the solution to the problem
\begin{align}\label{12}
\left\{\begin{aligned}
A^0_h:D^2 u_0^h &=f\quad\text{in }\Omega,\\
\hfill u_0^h &=0\quad \text{on }\partial\Omega,
\end{aligned}\right.
\end{align}
and $u_0\in H^2(\Omega)\cap H^1_0(\Omega)$ is the solution to the homogenized problem \eqref{1,4}.
\end{lemma}

\begin{proof}
We let $w_h:=u_0-u_0^h\in H^2(\Omega)\cap H^1_0(\Omega)$ and note that $w_h$ is the unique solution to the boundary-value problem
\begin{align*}
\left\{\begin{aligned}
A^0:D^2 w_h &=(A_h^0-A^0):D^2 u_0^h\quad\text{in }\Omega,\\
\hfill w_h &=0\quad\quad\quad\quad\quad\quad\quad\quad \text{on }\partial\Omega.
\end{aligned}\right.
\end{align*}
We recall that $A^0\in \R^{n\times n}$ is an elliptic constant matrix. For $h>0$ sufficiently small, by an $H^2$ a priori estimate, the Cauchy--Schwarz inequality and Lemma 3.1,
\begin{align*}
\|w_h\|_{H^2(\Omega)}&\lesssim \|(A_h^0-A^0):D^2u_0^h\|_{L^2(\Omega)}\\
&\lesssim \left( \int_{\Omega} \left\lvert \sum_{i,j=1}^n (a_{ij,h}^0-a_{ij}^0)\partial^2_{ij}u_0^h\right\rvert^2\right)^{\frac{1}{2}}\\
&\lesssim \left(\int_{\Omega} \left(\sum_{i,j=1}^n\lvert a_{ij,h}^0-a_{ij}^0\rvert^2\right)\left(\sum_{i,j=1}^n\lvert \partial^2_{ij}u_0^h\rvert^2\right)\right)^{\frac{1}{2}}\\
&\lesssim h\, \lvert u_0^h\rvert_{H^2(\Omega)}.
\end{align*}
Finally, we show that for $h>0$ sufficiently small, we have
\begin{align}\label{13}
\|u_0^h\|_{H^2(\Omega)}\lesssim \|f\|_{L^2(\Omega)}
\end{align}
with the constant being independent of $h$. This can be seen by rewriting \eqref{12} as 
\begin{align}\label{13,1}
\left\{\begin{aligned}
A^0 : D^2 u_0^h &= f +(A^0-A^0_h):D^2 u_0^h \quad\text{in }\Omega,\\
\hfill u_0^h &= 0 \hspace{3.9cm}\text{ on }\partial\Omega.
\end{aligned}\right.
\end{align}
Then, again by an $H^2$ a priori estimate and Lemma 3.1,
\begin{align*}
\|u_0^h\|_{H^2(\Omega)}\lesssim \| f +(A^0-A^0_h):D^2 u_0^h\|_{L^2(\Omega)}\lesssim \| f\|_{L^2(\Omega)}+h\|u_0^h\|_{H^2(\Omega)}
\end{align*}
with constants independent of $h$, i.e., for $h>0$ sufficiently small, \eqref{13} holds with the constant being independent of $h$. 
\end{proof}

Finally, we can use an $H^1_0(\Omega)$-conforming finite element approximation $u_0^{h,k}$ to the solution $u_0^h$ of \eqref{12}, satisfying the error bound
\begin{align*}
\left\|u_0^h-u_0^{h,k}\right\|_{H^1(\Omega)}\lesssim k\|u_0^h\|_{H^2(\Omega)}\lesssim k \|f\|_{L^2(\Omega)}
\end{align*}
with constants independent of $h$. By the triangle inequality and the results obtained in this section, we have the following approximation result for $u_0$.

\begin{theorem}[$H^1$ approximation of $u_0$]
Assume either that $(\Omega,A,f)\in \calG^{0,2}$, or that $(\Omega,A,f)\in \calH^{0}$. Then, the approximation $u_0^{h,k}$ obtained by the procedure described above satisfies
\begin{align*}
\left\|u_0-u_0^{h,k}\right\|_{H^1(\Omega)}\lesssim (h+k) \|f\|_{L^2(\Omega)}.
\end{align*}
\end{theorem}

Let us now assume either that $(\Omega,A,f)\in \calG^{1,2}$ or that $(\Omega,A,f)\in \calH^{1}$. Further, assume that for $h>0$ sufficiently small, we have that $u_0^h\in H^3(\Omega)$ with 
\begin{align}\label{pbefore2}
\left\|u_0^h\right\|_{H^3(\Omega)}\lesssim \left\|f\right\|_{H^1(\Omega)},
\end{align}
where the constant is independent of $h$. The following lemma provides two situations where this is satisfied.

\begin{lemma}
Let $(\Omega,A,f)$ be such that
\begin{itemize}
\item[$(i)$] $(\Omega,A,f)\in \calG^{1,2}$ with $\partial\Omega\in C^3$, or 
\item[$(ii)$] $(\Omega,A,f)\in \calH^{1}$ with $\Omega\subset\R^2$ being a polygon and $f\in H^1_0(\Omega)$. 
\end{itemize}
Then, for $h>0$ sufficiently small, \eqref{pbefore2} holds.
\end{lemma}

Before we prove Lemma 3.3, we need the following result on the regularity of solutions to Poisson's problem on convex polygons, see also \cite{Gri11,HO14,HOS15,KWY16}.

\begin{lemma}
Let $\Omega\subset\R^2$ be a convex polygonal domain and $f\in H^1_0(\Omega)$. Then the solution $u\in H^1_0(\Omega)$ to the problem
\begin{align*}
\left\{\begin{aligned}
\Delta u &= f  \quad\text{in }\Omega,\\
\hfill u &= 0 \quad\text{on }\partial \Omega,
\end{aligned}\right.
\end{align*} 
satisfies the bound
\begin{align}\label{before fact from}
\|u\|_{H^3(\Omega)}\lesssim \|f\|_{H^1(\Omega)}.
\end{align}
\end{lemma}

\begin{proof}
First, note that since $\Omega\subset\R^2$ is a convex polygonal domain, we have $u\in H^2(\Omega)\cap H^1_0(\Omega)$ with $\|u\|_{H^2(\Omega)}\lesssim \|f\|_{L^2(\Omega)}$, see \cite{Gri11}. Since $f\in H^1_0(\Omega)$, there exists a sequence of smooth functions with compact support $(f_m)_m\subset C_c^{\infty}(\Omega)$ such that $f_m\rightarrow f$ in $H^1(\Omega)$. Let $(u_m)_m\subset H^1_0(\Omega)$ be the sequence of solutions in $H^1_0(\Omega)$ to $\Delta u_m = f_m$ in $\Omega$, and note that $(u_m)_m\subset C^{\infty}(\bar{\Omega})$ since the functions $f_m$ satisfy compatibility conditions of any order, see \cite[Sec. 5.1]{Gri11}. Again we use the $H^2$-regularity result for solutions of Poisson's problem on convex polygons to obtain
\begin{align*}
\|u_m-u\|_{H^2(\Omega)} \lesssim \|f_m - f\|_{L^2(\Omega)}\rightarrow 0,
\end{align*}
i.e., $u_m\rightarrow u$ in $H^2(\Omega)$. 

Next, we shall use the fact that
\begin{align}\label{fact from KWY}
\lvert v\rvert_{H^3(\Omega)} = \|\nabla(\Delta v)\|_{L^2(\Omega)}\quad\forall\, v\in \left\{w\in H^1_0(\Omega):\Delta w\in H^1_0(\Omega)\right\}\cap C^{\infty}(\bar{\Omega}),
\end{align}
see \cite{KWY16}. We apply \eqref{fact from KWY} to the difference of two elements of the sequence $(u_m)_m$ to find that $(u_m)_m$ is a Cauchy sequence in $H^3(\Omega)$, using that $f_m\rightarrow f$  in $H^1(\Omega)$. Thus, $u_m\rightarrow u$ in $H^3(\Omega)$ and passing to the limit in \eqref{fact from KWY} applied to the functions $u_m$ yields
\begin{align*}
\lvert u\rvert_{H^3(\Omega)} = \|\nabla f\|_{L^2(\Omega)}.
\end{align*}   
Since $\|u\|_{H^2(\Omega)}\lesssim \|f\|_{L^2(\Omega)}$, we conclude the bound \eqref{before fact from}.
\end{proof}

\begin{remark}
The assumption $f\in H^1_0(\Omega)$ in Lemma 3.4 can be weakened provided $f$ satisfies certain compatibility conditions, see \cite[Theorem 5.1.2.4]{Gri11}.
\end{remark}

Now we are in a position to prove Lemma 3.3, using standard elliptic regularity theory, Lemma 3.4, and a scaling argument. 

\begin{proof}[Proof of Lemma 3.3]
We start with the case $(i)$. To this end, let $(\Omega,A,f)\in \calG^{1,2}$ with $\partial\Omega\in C^3$. Then, by elliptic regularity theory, we have $u_0^h\in H^3(\Omega)$. Using elliptic regularity for problem \eqref{13,1} yields
\begin{align*}
\|u_0^h\|_{H^3(\Omega)}\lesssim \| f +(A^0-A^0_h):D^2 u_0^h\|_{H^1(\Omega)}\lesssim \| f\|_{H^1(\Omega)}+h\|u_0^h\|_{H^3(\Omega)}
\end{align*}
with constants independent of $h$, i.e., for $h>0$ sufficiently small, \eqref{pbefore2} holds with the constant being independent of $h$.

Let us now show the claim for the case $(ii)$. To this end, let $(\Omega,A,f)\in \calH^{1}$ with $\Omega\subset\R^2$ being a polygon and $f\in H^1_0(\Omega)$. Since 
\begin{align*}
A_h^0 = A^0 + \left(A_h^0 - A^0 \right) =: A^0 + B_h
\end{align*}
is symmetric and elliptic for $h>0$ sufficiently small, there exists an orthogonal matrix $Q_h\in \R^{2\times 2}$ with $Q_hQ_h^{\mathrm{T}}=Q_h^{\mathrm{T}}Q_h=I_2$ such that
\begin{align*}
Q_h \left(A^0 + B_h \right)Q_h^{\mathrm{T}} = \mathrm{diag}(\lambda_h^{+},\lambda_h^{-})=:\Lambda_h,
\end{align*}
where $\lambda_h^{\pm}>0$ are given by
\begin{align*}
2\lambda_h^{\pm}=\mathrm{tr}\left(A^0 + B_h \right)\pm\left(\left(\mathrm{tr}\left(A^0 + B_h \right)\right)^2-4\det\left(A^0 + B_h \right) \right)^{\frac{1}{2}}.
\end{align*}
We note that, by Lemma 3.1, the entries of $B_h = (b_{ij}^h)_{1\leq i,j\leq 2}$ satisfy $b_{ij}^h \lesssim h$, and therefore, for $h>0$ sufficiently small, we have $0<\lambda_h^{\pm}+(\lambda_h^{\pm})^{-1}\lesssim 1$.

The problem \eqref{12} in the new coordinates reads
\begin{align}\label{problem in new coordinates}
\begin{split}
\left\{\begin{aligned}
\Delta U_h &= F_h\quad\text{in }P_h,\\
U_h &= 0\quad\;\text{ on }\partial P_h,
\end{aligned}\right.
\end{split}
\end{align}
where $U_h:= u_0^h\left(Q_h^{\mathrm{T}} \Lambda_h^{\frac{1}{2}} \;\cdot\; \right)$, $F_h := f\left(Q_h^{\mathrm{T}} \Lambda_h^{\frac{1}{2}} \;\cdot\; \right)$, and $P_h:=\Lambda_h^{-\frac{1}{2}}Q_h\Omega$. Note that $P_h$ is still a bounded convex polygonal domain and that $F_h\in H^1_0(P_h)$. By the change of variables formula and the orthogonality of $Q_h$,
\begin{align*}
\|f\|_{H^1(\Omega)}^2 = \int_{\Omega}\left( \lvert f\rvert^2 + \left\lvert \nabla f\right\rvert^2\right) &= \det \Lambda_h^{\frac{1}{2}} \int_{P_h}\left( \left\lvert f\left(Q_h^{\mathrm{T}} \Lambda_h^{\frac{1}{2}} \;\cdot\; \right)\right\rvert^2 + \left\lvert \nabla f\left(Q_h^{\mathrm{T}} \Lambda_h^{\frac{1}{2}} \;\cdot\; \right)\right\rvert^2\right)\\
&= \det \Lambda_h^{\frac{1}{2}} \int_{P_h}\left( \left\lvert F_h\right\rvert^2 + \left\lvert Q_h^{\mathrm{T}} \Lambda_h^{-\frac{1}{2}}\nabla F_h\right\rvert^2\right)\\
&= \det \Lambda_h^{\frac{1}{2}} \int_{P_h}\left( \left\lvert F_h\right\rvert^2 + \left\lvert  \Lambda_h^{-\frac{1}{2}}\nabla F_h\right\rvert^2\right)\\
&\gtrsim \int_{P_h}\left( \left\lvert F_h\right\rvert^2 + \left\lvert  \nabla F_h\right\rvert^2\right) = \left\|F_h\right\|_{H^1(P_h)}^2.
\end{align*}
Using Lemma 3.4, we have that, for $h>0$ sufficiently small, the solution to \eqref{problem in new coordinates} satisfies
\begin{align*}
\|U_h\|_{H^3(P_h)}\lesssim \left\|F_h\right\|_{H^1(P_h)}\lesssim \|f\|_{H^1(\Omega)}
\end{align*}
with constants independent of $h$. It remains to show the bound 
\begin{align}\label{establish this bound}
\|u_0^h\|_{H^3(\Omega)}\lesssim \|U_h\|_{H^3(P_h)}.
\end{align}
By the change of variables formula and the orthogonality of $Q_h$, we obtain similarly as before,
\begin{align*}
\|u_0^h\|_{H^3(\Omega)}^2 &= \int_{\Omega}\left( \lvert u_0^h\rvert^2 + \left\lvert \nabla u_0^h\right\rvert^2 + \left\lvert D^2 u_0^h\right\rvert^2\right) +\sum_{i=1}^2 \int_{\Omega} \left\lvert D^2 \partial_i u_0^h\right\rvert^2 \\
&= \det \Lambda_h^{\frac{1}{2}} \int_{P_h}\left( \left\lvert U_h\right\rvert^2 + \left\lvert Q_h^{\mathrm{T}} \Lambda_h^{-\frac{1}{2}}\nabla U_h\right\rvert^2 + \left\lvert Q_h^{\mathrm{T}} \Lambda_h^{-\frac{1}{2}}\;D^2 U_h\; \Lambda_h^{-\frac{1}{2}}Q_h\right\rvert^2\right) \\& \quad\quad\quad+ \sum_{i=1}^2 \det \Lambda_h^{\frac{1}{2}}\int_{P_h}\left\lvert\sum_{j=1}^2   \frac{(Q_h)_{ji}}{\sqrt{(\Lambda_h)_{jj}}}  Q_h^{\mathrm{T}} \Lambda_h^{-\frac{1}{2}}\;D^2\partial_j U_h\; \Lambda_h^{-\frac{1}{2}}Q_h\right\rvert^2 \\
&=  \det \Lambda_h^{\frac{1}{2}} \int_{P_h}\left( \left\lvert U_h\right\rvert^2 + \left\lvert  \Lambda_h^{-\frac{1}{2}}\nabla U_h\right\rvert^2 + \left\lvert  \Lambda_h^{-\frac{1}{2}}\;D^2 U_h\; \Lambda_h^{-\frac{1}{2}}\right\rvert^2\right) \\& \quad\quad\quad+ \sum_{i=1}^2 \frac{\det \Lambda_h^{\frac{1}{2}}}{(\Lambda_h)_{ii}} \int_{P_h}\left\lvert   \Lambda_h^{-\frac{1}{2}}\;D^2\partial_i U_h\; \Lambda_h^{-\frac{1}{2}}\right\rvert^2\\
&\lesssim \int_{P_h} \left(\left\lvert U_h\right\rvert^2 + \left\lvert  \nabla U_h\right\rvert^2+ \left\lvert D^2 U_h\right\rvert^2\right)+ \sum_{i=1}^2 \int_{P_h} \left\lvert D^2 \partial_i U_h\right\rvert^2 = \|U_h\|_{H^3(P_h)}^2,
\end{align*}
i.e., we have established the bound \eqref{establish this bound}. We conclude that, for $h>0$ sufficiently small, we have \eqref{pbefore2}, i.e.,
\begin{align*}
\|u_0^h\|_{H^3(\Omega)}\lesssim \|f\|_{H^1(\Omega)},
\end{align*}
where the constant is independent of $h$. 
\end{proof}

Then an $H^2(\Omega)\cap H^1_0(\Omega)$-conforming finite element approximation $u_0^{h,k}$ to the solution $u_0^h$ of \eqref{12}, that satisfies the error bound
\begin{align}\label{p2}
\left\|u_0^h-u_0^{h,k}\right\|_{H^2(\Omega)}\lesssim k \left\|u_0^h\right\|_{H^3(\Omega)}\lesssim k \left\|f\right\|_{H^1(\Omega)},
\end{align}
provides by Lemma 3.2 and the triangle inequality an approximation to $u_0$.

\begin{theorem}[$H^2$-norm approximation of $u_0$]
Assume either that $(\Omega,A,f)\in \calG^{1,2}$ or that $(\Omega,A,f)\in \calH^{1}$, and assume \eqref{pbefore2}. Then, the approximation $u_0^{h,k}$ obtained by the procedure described above satisfies
\begin{align*}
\left\|u_0-u_0^{h,k}\right\|_{H^2(\Omega)}\lesssim (h+k) \|f\|_{H^1(\Omega)}.
\end{align*}
\end{theorem}

\begin{remark}[Improvements]
We note that if we assume that $A\in W^{2,\infty}(Y)$, then we have the following improved results.
\begin{itemize}
\item[$(i)$] Approximation of $m$: In this case, $m\in H^2(Y)$ and we have that
\begin{align*}
\inf_{\tilde{v}_h\in  \tilde{M}_h} \|m-(\tilde{v}_h+1)\|_{H^1(Y)} \leq \left\|m-\mathcal{I}_h m - \int_Y (m- \mathcal{I}_h m)\right\|_{H^1(Y)}\lesssim h \|m\|_{H^2(Y)},
\end{align*}
by choosing $\tilde{v}_h= \mathcal{I}_h m - \int_Y \mathcal{I}_h m$, and using an interpolation error bound. Therefore, Theorem 3.1 yields
\begin{align*}
\|m-m_h\|_{L^2(Y)}+h\|m-m_h\|_{H^1(Y)}\lesssim h^2 \|m\|_{H^2(Y)}.
\end{align*}
\item[$(ii)$] Approximation of $A^0$: By an interpolation error bound and the fact that $m_h$ is piecewise linear, one has
\begin{align*}
\|a_{ij}m_h-\mathcal{I}_h(a_{ij}m_h)\|_{L^1(Y)} \lesssim  h^2\|a_{ij}\|_{W^{2,\infty}(Y)}\|m_h\|_{H^1(Y)}.
\end{align*}
Therefore, the proof of Lemma 3.1 yields
\begin{align*}
\max_{1\leq i,j\leq n}\;\left\lvert a_{ij}^0-a_{ij,h}^0\right\rvert \lesssim h^2 \|A\|_{W^{2,\infty}(Y)}\|m\|_{H^2(Y)} \lesssim h^2 \|A\|_{W^{2,\infty}(Y)} .
\end{align*}
\item[$(iii)$] Approximation of $u_0$: It follows that the results of Lemma 3.2, Theorem 3.2 and Theorem 3.3 can be improved to second-order convergence in $h$, i.e.,
\begin{align*}
\left\|u_0-u_0^{h,k}\right\|_{H^s(\Omega)}\lesssim(h^2\|A\|_{W^{2,\infty}(Y)}+k) \|f\|_{H^{s-1}(\Omega)} = \mathcal{O}(h^2+k),
\end{align*}
for $s=1,2$, respectively.
\end{itemize}
\end{remark}

For the approximation of derivatives of $u_0$ of higher than second order, the post-processing method of Babu\v{s}ka in \cite{Bab70} can be used to obtain error bounds in norms involving derivatives of higher order than the energy norm (the norm natural to the problem).

For bounded convex polygonal domains $\Omega\subset \R^2$, an $H^2$-conforming approximation to the solution of \eqref{12} can be obtained as follows. Assume that $f\in H^1_0(\Omega)$ so that \eqref{pbefore2} holds. Consider a shape-regular triangulation of $\Omega$ into triangles with longest edge $k>0$, and let 
\begin{align*}
V_k\subset H^2(\Omega)\cap H^1_0(\Omega)
\end{align*}
be an appropriate finite element space. In practice, the Hsieh--Clough--Tocher element and the Argyris element can be used as $H^2$-conforming elements. Then, for $h>0$ sufficiently small, standard finite element analysis can be used to show that there is a unique function $u_0^{h,k}\in V_k$ such that
\begin{align}\label{q1}
\int_{\Omega} \left(A^0_h:D^2 u_0^{h,k}\right)\left( A^0_h:D^2 \fhi_k\right)  = \int_{\Omega}f\left( A^0_h:D^2 \fhi_k\right)\quad\quad\forall\,\fhi_k\in V_k,
\end{align} 
and that the error bound \eqref{p2} holds.

\subsection{Approximation of the Corrector}

We now address problem \eqref{6,8} and present a method for $A\in W^{2,\infty}(Y)$. To simplify the notation and the arguments, we assume that we know the invariant measure $m$ and the matrix $A^0=(a_{ij}^0)_{1\leq i,j\leq n}$ exactly instead of working with our approximation $A_h^0$. 

For a given $Y$-periodic right-hand side $g\in W^{2,\infty}(Y)$, we address the problem 
\begin{align*}
\begin{cases}
-\nabla\cdot(A\nabla\chi)+(\div A)\cdot\nabla \chi = -g \quad\text{in }Y,\\
\chi\;\text{is $Y$-periodic},\; \int_Y \chi=0.
\end{cases}
\end{align*}
Obtaining an approximation for second-order derivatives via finite elements is not straightforward since the natural solution space is $W_{\text{per}}(Y)$. We present a method of successively approximating higher derivatives. 

Let $\chi_h$ be a $W_{\text{per}}(Y)$-conforming finite element approximation to $\chi$, i.e.,
\begin{align*}
\chi_h\in V_h,\quad \int_Y A\nabla \chi_h\cdot\nabla\fhi +\int_Y \fhi\;(\div A)\cdot \nabla \chi_h = -\int_Y g\fhi\quad\quad\forall\, \fhi\in V_h,
\end{align*} 
with $V_h\subset W_{\text{per}}(Y)$ finite-dimensional, and satisfying the error estimate
\begin{align*}
\|\chi_h-\chi\|_{H^1(Y)}\lesssim h.
\end{align*}
Let $r\in \{1,\dots,n\}$ and write $\xi_r:=\partial_r \chi$. Then, using the equation
\begin{align*}
-\nabla \cdot (A\nabla \chi)+(\div A)\cdot\nabla \chi = -g\quad\text{in }Y,
\end{align*}
we find that weakly, there holds
\begin{align*}
-\nabla \cdot (A\nabla \xi_r)+(\div A)\cdot\nabla \xi_r = -\partial_k g+\nabla\cdot(\partial_k A\;\nabla \chi)- \left(\div(\partial_k A)\right)\cdot\nabla \chi \quad\text{in }Y.
\end{align*}
Further, we claim that $\xi_r\in W_{\text{per}}(Y)$. Indeed, this follows from the regularity and periodicity of $\chi$ and 
\begin{align*}
\int_Y \partial_r \chi = \int_{\partial Y} \chi \nu\cdot e_r=0.
\end{align*}
Therefore, $\xi_r\in W_{\text{per}}(Y)$ satisfies
\begin{align*}
\begin{cases}
-\nabla \cdot (A\nabla \xi_r)+(\div A)\cdot\nabla \xi_r = -\partial_r g+\nabla\cdot(\partial_r A\;\nabla \chi)- \left(\div(\partial_r A)\right)\cdot\nabla \chi \quad\text{in }Y,\\
\xi_r\;\text{is $Y$-periodic},\; \int_Y \xi_r=0.
\end{cases}
\end{align*}
Now we use our $H^1$-conforming approximation for $\chi$ for the right-hand side and use a $W_{\text{per}}(Y)$-conforming finite element method for approximating the solution $v\in W_{\text{per}}(Y)$ to the following problem:
\begin{align}\label{14}
\begin{cases}
-\nabla \cdot (A\nabla v)+(\div A)\cdot\nabla v = -\partial_r g+\nabla\cdot(\partial_r A\;\nabla \chi_h)- \left(\div(\partial_r A)\right)\cdot\nabla \chi_h -c\quad\text{in }Y,\\
v\;\text{is $Y$-periodic},\; \int_Y v=0,
\end{cases}
\end{align}
where $c$ is such that this problem admits a unique solution (such that the solvability condition \eqref{solvcond} is satisfied). By looking at the problem for $v-\xi_r$, one obtains the comparison result
\begin{align*}
\|v-\xi_r\|_{H^1(Y)}&\lesssim \|\nabla\cdot(\partial_r A\; \nabla(\chi_h-\chi))\|_{W_{\text{per}}(Y)'}+\|\left(\div(\partial_r A)\right)\cdot\nabla (\chi_h-\chi)\|_{W_{\text{per}}(Y)'}\\ &\lesssim \|A\|_{W^{2,\infty}(Y)}\|\chi_h-\chi\|_{H^1(Y)}\\&\lesssim h\|A\|_{W^{2,\infty}(Y)} = \mathcal{O}(h) .
\end{align*}
Let $v_h$ be a $W_{\text{per}}(Y)$-conforming finite element approximation of \eqref{14} satisfying
\begin{align*}
\|v_h-v\|_{H^1(Y)}\leq C h
\end{align*}
for some constant $C=C(\|A\|_{W^{2,\infty}(Y)})>0$. Then, using the triangle inequality, we obtain
\begin{align*}
\|v_h-\xi_r\|_{H^1(Y)}\leq C h
\end{align*}
for some constant $C=C(\|A\|_{W^{2,\infty}(Y)})>0$. Using this procedure for $r=1,\dots,n$, we eventually obtain approximations to derivatives of order up to two of $\chi$.

\subsection{Approximation of $u_{\eps}$}

We assume either that $(\Omega,A,f)\in \calG^{2,2}$ or that $(\Omega,A,f)\in \calH^{2}$. Let $n\in\{2,3\}$, $\eps\in (0,1]$, and assume that
\begin{align*}
u_0\in H^4(\Omega).
\end{align*} 
Then we know that \eqref{addregu0} is satisfied, and by Theorem 2.5 we have that
\begin{align}\label{new1}
\left\|u_{\eps}-u_0-\eps^2\sum_{i,j=1}^n \chi_{ij}\left(\frac{\cdot}{\eps}\right)\partial^2_{ij} u_0\right\|_{H^2(\Omega)}\lesssim \sqrt{\eps}\;\|u_0\|_{W^{2,\infty}(\Omega)} + \eps\|u_0\|_{H^4(\Omega)},
\end{align}
where $u_0$ is the solution to the homogenized problem, and $\chi_{ij}$ are the corrector functions given as the solutions to \eqref{6,8}. This result can be used to construct an approximation of $u_{\eps}$, i.e., to the solution of problem \eqref{1} for small $\eps$. We note that \eqref{new1} implies that
\begin{align}\label{new2}
\begin{split}
&\|u_{\eps}-u_0\|_{H^1(\Omega)}+\sum_{k,l=1}^n \left\|\partial_{kl}^2u_{\eps}-\left(\partial_{kl}^2u_0 + \sum_{i,j=1}^n \left(\partial_{kl}^2\chi_{ij}\right)\left(\frac{\cdot}{\eps}\right)\partial^2_{ij} u_0\right)\right\|_{L^2(\Omega)}\\ &\lesssim \sqrt{\eps}\;\|u_0\|_{W^{2,\infty}(\Omega)} + \eps\|u_0\|_{H^4(\Omega)}.
\end{split}
\end{align}
This leads to the following approximation result for $u_{\eps}$.

\begin{theorem}[Approximation of $u_{\eps}$]
In the situation described above, let $(u_{0,h})_{h>0}\subset H^2(\Omega)$ be a family of $H^2$-conforming approximations for $u_0$ satisfying the error bound
\begin{align*}
\|u_0-u_{0,h}\|_{H^2(\Omega)}\lesssim h \|f\|_{H^1(\Omega)},
\end{align*}
and for $1\leq i,j,k,l\leq n$, let $(z_{ij,h}^{kl})_{h>0}\subset L^2_{\text{per}}(Y)$ be a family of $L^2$ approximations for $\partial_{kl}^2\chi_{ij}$ satisfying the error bound
\begin{align*}
\|\partial_{kl}^2\chi_{ij}-z_{ij,h}^{kl}\|_{L^2(Y)}\lesssim h.
\end{align*}
Then, by writing
\begin{align*}
u_{\eps,h}^{kl}:=\partial_{kl}^2u_{0,h} + \sum_{i,j=1}^n z_{ij,h}^{kl}\left(\frac{\cdot}{\eps}\right)\partial^2_{ij} u_{0,h},
\end{align*}
we have that
\begin{align*}
&\|u_{\eps}-u_{0,h}\|_{H^1(\Omega)}+\sum_{k,l=1}^n \left\|\partial_{kl}^2u_{\eps}-u_{\eps,h}^{kl}\right\|_{L^1(\Omega)}\\&\lesssim\left(\sqrt{\eps}+h\right)\|u_0\|_{W^{2,\infty}(\Omega)} + \eps\|u_0\|_{H^4(\Omega)} + h\|f\|_{H^1(\Omega)}   .
\end{align*}
\end{theorem}
\begin{proof}
We use \eqref{new2} and the triangle inequality to obtain
\begin{align*}
\|u_{\eps}-u_{0,h}\|_{H^1(\Omega)}&\leq \|u_{\eps}-u_0\|_{H^1(\Omega)}+ \|u_0-u_{0,h}\|_{H^1(\Omega)}\\&\lesssim\sqrt{\eps}\;\|u_0\|_{W^{2,\infty}(\Omega)} + \eps\|u_0\|_{H^4(\Omega)}+ h \|f\|_{H^1(\Omega)},
\end{align*}
and for $1\leq k,l\leq n$,
\begin{align*}
\left\|\partial_{kl}^2u_{\eps}-u_{\eps,h}^{kl}\right\|_{L^1(\Omega)}&\lesssim\sqrt{\eps}\;\|u_0\|_{W^{2,\infty}(\Omega)} + \eps\|u_0\|_{H^4(\Omega)}+ h \|f\|_{H^1(\Omega)}\\ &\quad+\sum_{i,j=1}^n\left\| \left(\partial_{kl}^2\chi_{ij}\right)\left(\frac{\cdot}{\eps}\right)\partial^2_{ij} u_0 -  z_{ij,h}^{kl}\left(\frac{\cdot}{\eps}\right)\partial^2_{ij} u_{0,h}   \right\|_{L^1(\Omega)}.
\end{align*}
It remains to study the last term on the right-hand side of the above inequality. For fixed $1\leq i,j\leq n$, we use again the triangle inequality to obtain
\begin{align*}
&\left\| \left(\partial_{kl}^2\chi_{ij}\right)\left(\frac{\cdot}{\eps}\right)\partial^2_{ij} u_0 -  z_{ij,h}^{kl}\left(\frac{\cdot}{\eps}\right)\partial^2_{ij} u_{0,h}   \right\|_{L^1(\Omega)}\\ &\leq \left\|  z_{ij,h}^{kl}\left(\frac{\cdot}{\eps}\right)\left(\partial^2_{ij} u_{0}-\partial^2_{ij} u_{0,h}  \right) \right\|_{L^1(\Omega)}+\left\| \left(\partial_{kl}^2\chi_{ij}-z_{ij,h}^{kl}\right)\left(\frac{\cdot}{\eps}\right) \partial^2_{ij} u_{0} \right\|_{L^1(\Omega)}\\
&\lesssim \left\|  z_{ij,h}^{kl}\left(\frac{\cdot}{\eps}\right)\right\|_{L^2(\Omega)}\|u_0-u_{0,h}\|_{H^2(\Omega)} + \left\| \left(\partial_{kl}^2\chi_{ij}-z_{ij,h}^{kl}\right)\left(\frac{\cdot}{\eps}\right) \right\|_{L^2(\Omega)}\|u_0\|_{W^{2,\infty}(\Omega)}\\
&\lesssim  h\left(\left\|  z_{ij,h}^{kl}\left(\frac{\cdot}{\eps}\right)\right\|_{L^2(\Omega)}  \|f\|_{H^1(\Omega)} +  \|u_0\|_{W^{2,\infty}(\Omega)}\right).
\end{align*}
In the last step, we used that by the transformation formula and periodicity (cover $\Omega/\eps$ by $\mathcal{O}(\eps^{-n})$ many cells of unit length), there holds
\begin{align}\label{new3}
\left\| \left(\partial_{kl}^2\chi_{ij}-z_{ij,h}^{kl}\right)\left(\frac{\cdot}{\eps}\right) \right\|_{L^2(\Omega)}\lesssim \left\| \partial_{kl}^2\chi_{ij}-z_{ij,h}^{kl} \right\|_{L^2(Y)}\lesssim h.
\end{align}
We claim that
\begin{align*}
\left\|  z_{ij,h}^{kl}\left(\frac{\cdot}{\eps}\right)\right\|_{L^2(\Omega)}\lesssim h+1.
\end{align*}
Indeed, we use the triangle inequality, \eqref{new3} and the fact that $\chi_{ij}\in W^{2,\infty}(Y)$ to obtain
\begin{align*}
\left\|  z_{ij,h}^{kl}\left(\frac{\cdot}{\eps}\right)\right\|_{L^2(\Omega)}&\leq \left\| \left(\partial_{kl}^2\chi_{ij}-z_{ij,h}^{kl}\right)\left(\frac{\cdot}{\eps}\right) \right\|_{L^2(\Omega)}+\left\|\partial_{kl}^2\chi_{ij}  \right\|_{L^{\infty}(Y)}\lesssim h+1.
\end{align*}
\end{proof}

The approximations of $u_0$ and the corrector functions can be obtained as described in Section 3.1 and 3.2. Let us conclude this section by remarking that if the second derivatives of the corrector functions are approximated in the space $L^{\infty}(Y)$ or if the solution to the homogenized problem is approximated in the space $W^{2,\infty}(\Omega)$, then one obtains by a similar proof an approximation result for the second derivatives of $u_{\eps}$ in $L^2(\Omega)$.

\begin{remark}
If $(z_{ij,h}^{kl})_{h>0}\subset L^{\infty}_{\text{per}}(Y)$ is a family of $L^{\infty}$ approximations for $\partial_{kl}^2\chi_{ij}$ satisfying the error bound
\begin{align*}
\|\partial_{kl}^2\chi_{ij}-z_{ij,h}^{kl}\|_{L^{\infty}(Y)}=\mathcal{O}( h),
\end{align*}
and $(u_{0,h})_{h>0}$ is as in Theorem 3.4, then we have that
\begin{align*}
\|u_{\eps}-u_{0,h}\|_{H^1(\Omega)}+\sum_{k,l=1}^n \left\|\partial_{kl}^2u_{\eps}-u_{\eps,h}^{kl}\right\|_{L^2(\Omega)}=\mathcal{O}(\sqrt{\eps}+h).
\end{align*}
The same holds true when $(u_{0,h})_{h>0}\subset W^{2,\infty}(\Omega)$ is a family of $W^{2,\infty}$-conforming approximations for $u_0$ satisfying the error bound
\begin{align*}
\|u_0-u_{0,h}\|_{W^{2,\infty}(\Omega)}=\mathcal{O}(h),
\end{align*}
and $(z_{ij,h}^{kl})_{h>0}$ is as in Theorem 3.4.
\end{remark}

\subsection{Nonuniformly Oscillating Coefficients}

In this section, we discuss the case of nonuniformly oscillating coefficients, i.e., coefficients depending on $x$ and $\frac{x}{\eps}$.
We consider the problem
\begin{align}\label{Nonuniform osc}
\left\{\begin{aligned}
A\left(\cdot,\frac{\cdot}{\eps}\right):D^2 u_{\eps}&=f\quad\text{in }\Omega,\\
\hfill u_{\eps} &= 0\quad\text{on }\partial\Omega,
\end{aligned}\right.
\end{align}
where the triple $(\Omega,A,f)$ satisfies one of the following sets of assumptions.

\begin{definition}[Sets of assumptions $\calG,\calH$]
For $m\in \N_0$, we write
\begin{itemize}
\item[$(i)$] $(\Omega,A,f)\in \calG$ if and only if $\Omega\subset \R^n$ is a bounded $C^{2,\gamma}$ domain, $f\in L^2(\Omega)$, and $A=A^{\mathrm{T}}:\Omega\times \R^n\rightarrow\R^{n\times n}$ satisfies
\begin{align}\label{assnu}
\begin{cases}A=A(x,y)\in W^{2,\infty}(\Omega; W^{1,q}(Y) )\text{ for some }q\in(n,\infty],\\
A(x,\cdot)\text{ is } \text{$Y$-periodic},\\
\exists\; \lambda,\Lambda>0:\; \lambda \lvert \xi\rvert^2\leq  A(x,y)\xi\cdot \xi\leq \Lambda \lvert\xi\rvert^2\quad\forall\,\xi,y\in\R^n,\, x\in \Omega.
\end{cases}
\end{align} 
\item[$(ii)$] $(\Omega,A,f)\in \calH$ if and only if $\Omega\subset \R^n$ is a bounded convex domain, $f\in L^2(\Omega)$, and $A=A^{\mathrm{T}}:\Omega\times \R^n\rightarrow\R^{n\times n}$ satisfies \eqref{assnu} and 
\begin{align}\label{assnucordes}
\exists\;\delta\in (0,1]: \frac{\lvert A(x,y)\rvert^2}{(\mathrm{tr}A(x,y))^2}\leq \frac{1}{n-1+\delta}\quad\forall\, (x,y)\in \Omega\times \R^n.
\end{align}
\end{itemize}
\end{definition}

In view of Remark 2.1, we see that the Cordes condition \eqref{assnucordes} is always satisfied when $n=2$. Well-posedness to the problem \eqref{Nonuniform osc} is guaranteed by the following theorem, see \cite[Theorem 9.15]{GT01} and \cite[Theorem 3]{SS13}.

\begin{theorem}[Existence and uniqueness of strong solutions]
Assume either that $(\Omega,A,f)\in \calG$, or that $(\Omega,A,f)\in \calH$. Then, for any $\eps>0$, the problem \eqref{Nonuniform osc} admits a unique solution $u_{\eps}\in H^2(\Omega)\cap H^1_0(\Omega)$.
\end{theorem}

As in Section 2, uniform a priori estimates for the solution to \eqref{Nonuniform osc} allow passage to the limit in equation \eqref{Nonuniform osc}, see \cite{BBM86,BLP11}. The coefficient matrix of the homogenized problem now depends on the slow variable $x$, and is obtained by integrating against an invariant measure. Corrector results can then be shown as before. 

\begin{theorem}[Nonuniformly oscillating coefficients]
Assume that $\eps\in (0,1]$ and either that $(\Omega,A,f)\in \calG$, or that $(\Omega,A,f)\in \calH$. Then the following assertions hold.\begin{itemize}
\item[$(i)$] Uniform a priori estimate: The solution $u_{\eps}\in H^2(\Omega)\cap H^1_0(\Omega)$ to \eqref{Nonuniform osc} satisfies
\begin{align*}
\|u_{\eps}\|_{H^2(\Omega)}\lesssim \|f\|_{L^2(\Omega)}.
\end{align*}
\item[$(ii)$] Homogenization: The solution $u_{\eps}\in H^2(\Omega)\cap H^1_0(\Omega)$ to \eqref{Nonuniform osc} converges weakly in $H^2(\Omega)$ to the solution $u_0\in H^2(\Omega)\cap H^1_0(\Omega)$ of the homogenized problem
\begin{align}\label{nuohomequ}
\left\{
\begin{aligned}
A^0:D^2 u_{0}&=f\quad\text{in }\Omega,\\
\hfill u_{0} &= 0\quad\text{on }\partial\Omega,
\end{aligned}\right.
\end{align}
with $A^0:\Omega\rightarrow \R^{n\times n}$ given by
\begin{align*}
A^0(x):=\int_Y A(x,\cdot)m(x,\cdot),
\end{align*}
where $m=m(x,y)$ is the unique function $m:\Omega\times \R^n\rightarrow \R$ with $m\in C(\bar{\Omega}\times \R^n)$, $0<\bar{m}\leq m\leq M$ for some constants $\bar{m},M>0$, such that 
\begin{align*}
\begin{cases}
D^2:\left(A(x,\cdot)m(x,\cdot)\right)=0\quad\text{in }Y,\\
m(x,\cdot) \;\text{is $Y$-periodic},\;\int_Y m(x,\cdot)=1,
\end{cases}
\end{align*}
for any fixed $x\in \Omega$. The function $m$ is called the invariant measure.
\item[$(iii)$] Corrector estimate: Assume that $f\in H^2(\Omega)$ and $u_0\in H^4(\Omega)\cap W^{2,\infty}(\Omega)$.
Introducing the corrector function $\chi_{ij}$, $1\leq i,j\leq n$, as the solution to 
\begin{align*}
\begin{cases}
A(x,y):D_y^2 \chi_{ij}(x,y) = a_{ij}^0(x)-a_{ij}(x,y) ,\quad\,(x,y)\in \Omega\times Y,\\
\chi_{ij}(x,\cdot)\;\text{is $Y$-periodic},\;\int_Y \chi_{ij}(x,\cdot)=0,
\end{cases}
\end{align*}
we have that
\begin{align*}
\left\|u_{\eps}-u_0-\eps^2\sum_{i,j=1}^n \chi_{ij}\left(\cdot,\frac{\cdot}{\eps}\right)\partial^2_{ij} u_0\right\|_{H^2(\Omega)}\lesssim \sqrt{\eps}\|u_0\|_{W^{2,\infty}(\Omega)} + \eps\|u_0\|_{H^4(\Omega)}.
\end{align*}
\end{itemize}
\end{theorem} 

\begin{proof} 
$(i)$ For $(\Omega,A,f)\in \calH$, one shows similarly to the proof of \cite[Theorem 3]{SS13} and Theorem 2.2 that
\begin{align*}
\|u_{\eps}\|_{H^2(\Omega)}\lesssim \left\| \frac{\mathrm{tr}A\left(\cdot,\frac{\cdot}{\eps}\right)}{\lvert A\left(\cdot,\frac{\cdot}{\eps}\right)\rvert^2}\right\|_{L^{\infty}(\Omega)} \|f\|_{L^2(\Omega)}\lesssim \|f\|_{L^2(\Omega)}.
\end{align*}
For $(\Omega,A,f)\in \calG$, the claim follows from the method of freezing coefficients, using the uniform estimate from Theorem 2.2 for the operators $L_{x_0}:=A\left(x_0,\frac{\cdot}{\eps}\right):D^2$ for fixed $x_0\in \Omega$.\\

$(ii)$ The uniform estimate from $(i)$ yields weak convergence in $H^2(\Omega)$ and strong convergence in $H^1(\Omega)$ for a subsequence of $(u_{\eps})_{\eps>0}$ to some limit function $u_0\in H^2(\Omega)\cap H^1_0(\Omega)$. We multiply \eqref{Nonuniform osc} by $m\left(\cdot,\frac{\cdot}{\eps}\right)$ and follow the transformation performed in \cite{BBM86} to find that
\begin{align*}
m^{\eps}f = 2\,\nabla&\cdot \left(\tilde{A}^{\eps}\nabla u_{\eps}+\left[\div_x \tilde{A}\right]^{\eps}u_{\eps}\right) - 2\left[\div_x \tilde{A}\right]^{\eps}\cdot \nabla u_{\eps} \\&- \left[D_x^2:\tilde{A}\right]^{\eps}u_{\eps} - D^2:\left(\tilde{A}^{\eps}u_{\eps}\right)
\end{align*}
holds weakly, where $\tilde{A}:=Am$ and $v^{\eps}$ denotes $v\left(\cdot,\frac{\cdot}{\eps}\right)$. Passing to the limit, we obtain that $u_0$ is a weak solution of \eqref{nuohomequ}. We conclude the proof by noting that \eqref{nuohomequ} admits a unique strong solution, since $A^0$ is uniformly elliptic and Lipschitz continuous on $\bar{\Omega}$, see \cite{GT01,Gri11}.\\

$(iii)$ This can be proved similarly to Theorem 2.4 and Theorem 2.5, using that, by the assumptions made on $A$ and elliptic regularity, we have 
\begin{align*}
\chi_{kl}^{\eps},\, \left[\partial_{x_i} \chi_{kl}\right]^{\eps},\, \left[\partial_{y_i} \chi_{kl}\right]^{\eps},\, \left[\partial^2_{x_iy_j} \chi_{kl}\right]^{\eps},\,\left[\partial^2_{x_ix_j} \chi_{kl}\right]^{\eps}\in L^{\infty}(\Omega)
\end{align*}
for any $1\leq i,j,k,l\leq n$.
\end{proof}

Let us explain how the numerical scheme from Section 3.1 can be used for the numerical homogenization of \eqref{Nonuniform osc}. 

First, we consider a triangulation $\calT_k$ on $\bar{\Omega}$ consisting of nodes $\{x_i\}_{i\in I}$ with grid size $k>0$, and a triangulation $\calT_h$ on $Y$ with grid size $h>0$. Then, for any $i\in I$, we can use the scheme from Section 3.1 (see Theorem 3.1) to obtain an approximation $m^i_h\in H^1(Y)$ to $m_{x_i}=m(x_i,\cdot)$ such that
\begin{align*}
\|m_{x_i}-m_h^i\|_{L^2(Y)}+h\|m_{x_i}-m_h^i\|_{H^1(Y)}\lesssim h \inf_{\tilde{v}_h\in  \tilde{M}_h} \|m_{x_i}-(\tilde{v}_h+1)\|_{H^1(Y)}.
\end{align*} 
Further, we obtain that
\begin{align*}
A_h^{0,i}:=\int_Y \mathcal{I}_h\left( A(x_i,\cdot)\; m_h^i\right)
\end{align*}
is an approximation to $A^0(x_i)$ (see Lemma 3.1),
\begin{align}\label{l31}
\left\lvert A^0(x_i)-A_h^{0,i}\right\rvert \lesssim h.
\end{align}
Now we define $A_{h,k}^0$ to be a continuous piecewise linear function on $\Omega$ such that
\begin{align*}
A_{h,k}^0(x_i)=A_h^{0,i}.
\end{align*} 
Then, using \eqref{l31} and denoting the continuous piecewise linear interpolant of a function $\phi$ on the grid $\{x_i\}_{i\in I}$ by $\mathcal{I}_k\phi$, we have
\begin{align}\label{l32}
\begin{split}
\|A^0-A_{h,k}^0\|_{L^{\infty}(\Omega)}&\leq \|A^0-\mathcal{I}_k A^0\|_{L^{\infty}(\Omega)} + \| \mathcal{I}_k A^0- A_{h,k}^0\|_{L^{\infty}(\Omega)} \\ &\lesssim\|A^0-\mathcal{I}_k A^0\|_{L^{\infty}(\Omega)} + h.
\end{split} 
\end{align}

We observe that, similarly to the proof of Lemma 3.2, we obtain that the solution $u_0^{h,k}\in H^2(\Omega)\cap H^1_0(\Omega)$ to
\begin{align}\label{l33}
\left\{\begin{aligned}
A_{h,k}^0:D^2 u_0^{h,k} &=f\quad\text{in }\Omega,\\
\hfill u_0^{h,k} &=0\quad \text{on }\partial\Omega,
\end{aligned}\right.
\end{align}
satisfies, for $h,k>0$ sufficiently small,
\begin{align*}
\|u_0-u_0^{h,k}\|_{H^2(\Omega)}\lesssim \|A^0-A_{h,k}^0\|_{L^{\infty}(\Omega)}\|f\|_{L^2(\Omega)},
\end{align*}
and in view of \eqref{l32},
\begin{align*}
\|u_0-u_0^{h,k}\|_{H^2(\Omega)}\lesssim \left( \|A^0-\mathcal{I}_k A^0\|_{L^{\infty}(\Omega)} + h \right)\|f\|_{L^2(\Omega)},
\end{align*}
where $u_0$ is the solution to the homogenized problem \eqref{nuohomequ}. Finally, the solution to \eqref{l33} can be approximated by a standard finite element method on the triangulation $\calT_k$ which yields an approximation $u_{0,h,k}\in H^2(\Omega)\cap H^1_0(\Omega)$ to $u_0$ in the $H^2(\Omega)$-norm.

The approximation of $u_{\eps}$ can be obtained based on the corrector estimate from Theorem 3.6 analogously as in Section 3.3.

\section{Numerical Experiments}

\subsection{Problem with a Known $u_0$}

We consider the homogenization problem
\begin{align}\label{1265}
\left\{\begin{aligned}
A\left(\frac{\cdot}{\eps}\right):D^2 u_{\eps} &= f\quad\text{in }\Omega,\\
\hfill u_{\eps} &= 0\quad\text{on }\partial\Omega,
\end{aligned}\right.
\end{align}
on the domain 
\begin{align*}
\Omega:=Y=(0,1)^2,
\end{align*}
with the matrix-valued map
\begin{align*}
A:\R^2\rightarrow \R^{2\times 2},\quad A(y_1,y_2):=\begin{pmatrix}
1+\arcsin\left(\sin^2(\pi y_1)\right)  & \sin(\pi y_1)\cos(\pi y_1)\\
\sin(\pi y_1)\cos(\pi y_1) & 2+\cos^2(\pi y_1)
\end{pmatrix},
\end{align*}
and the right-hand side $f:\Omega\rightarrow\R$ to be specified below. We observe that the matrix-valued function $A$ satisfies \eqref{Aass} with $q=\infty$. Further, note that 
\begin{align*}
A(y)=(a_{ij}(y_1))_{1\leq i,j\leq 2}
\end{align*}
depends only on the first coordinate of $y=(y_1,y_2)\in\R^2$; see Figure 1.

\begin{figure}[H]
\includegraphics[scale=.8]{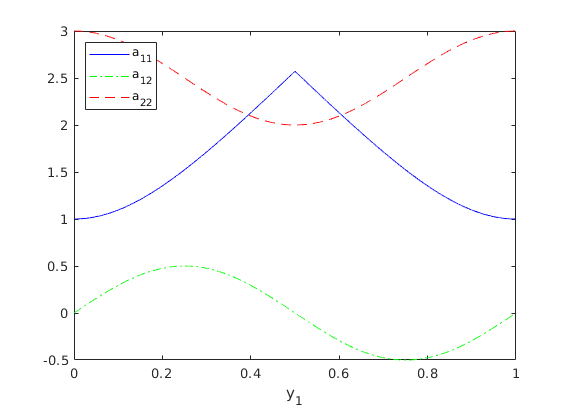}
\caption{The functions $a_{ij}(y_1)$ plotted on the interval $(0,1)$. }
\end{figure}

In this case we know that the homogenized problem is given by
\begin{align}\label{15}
\left\{\begin{aligned}
A^0:D^2 u_{0} &= f\quad\text{in }\Omega,\\
\hfill u_{0} &= 0\quad\text{on }\partial\Omega,
\end{aligned}\right.
\end{align}
where $A^0\in \R^{2\times 2}$ denotes the constant matrix
\begin{align*}
A^0=\int_Y Am
\end{align*}
with $m$ being the invariant measure 
\begin{align*}
m:\R^2\rightarrow\R,\quad m(y_1,y_2)=\left(\int_0^1 \frac{\mathrm{d}t}{a_{11}(t)}  \right)^{-1}\frac{1}{a_{11}(y_1)},
\end{align*}
see \cite{FO09}. Explicit computation yields that
\begin{align*}
a_{11}^0 &= \left(\int_0^1 \frac{\mathrm{d}t}{a_{11}(t)}  \right)^{-1}\hspace{2.3cm}\approx 1.4684,\\
a_{12}^0 &= \left(\int_0^1 \frac{\mathrm{d}t}{a_{11}(t)}  \right)^{-1}\int_0^1 \frac{a_{12}(t)}{a_{11}(t)}\;\mathrm{d}t =0,\\
a_{22}^0 &= \left(\int_0^1 \frac{\mathrm{d}t}{a_{11}(t)}  \right)^{-1}\int_0^1 \frac{a_{22}(t)}{a_{11}(t)}\;\mathrm{d}t \approx 2.6037.
\end{align*}
We consider the right-hand side given by
\begin{align*}
f:\Omega\rightarrow\R,\quad f(x_1,x_2) := a_{22}^0 x_1(x_1-1) + a_{11}^0 x_2(x_2-1).
\end{align*}
Then it is straightforward to check that the exact solution $u_0\in H^2(\Omega)\cap H^1_0(\Omega)$ to the homogenized problem \eqref{15} is given by
\begin{align*}
u_0:\Omega\rightarrow\R,\quad u_0(x_1,x_2) = \frac{1}{2}x_1(x_1-1)x_2(x_2-1).
\end{align*}
Note that we are in the situation $(\Omega,A,f)\in \calH^2$, that $f=0$ in the corners of $\Omega$ and that $u_0\in H^4(\Omega)$.

We use the scheme presented in Section 3.1 to approximate $m$, $A^0$ and $u_0$. We use the same mesh for approximating $m$ and $u_0$. The Hsieh--Clough--Tocher (HCT) element in FreeFem++ is used in the formulation \eqref{q1} for the $H^2$ approximation of $u_0$; see \cite{Hec12}. The gradient on the boundary is set to be the gradient of an $H^1$ approximation by $\mathbb{P}_2$ elements on a fine mesh.

\begin{figure}[H]
\setlength{\abovecaptionskip}{17.5 pt plus 4pt minus 2pt}
\mbox{
\subfigure{\includegraphics[scale=.5]{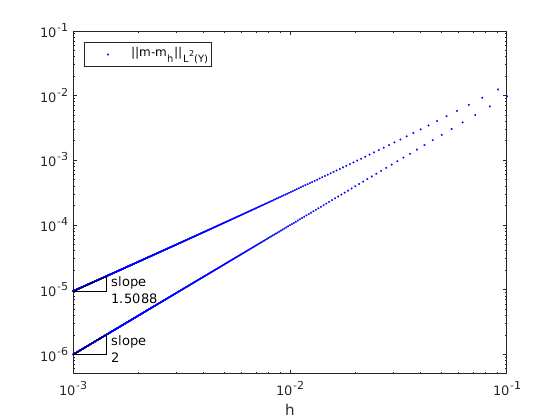}}\,
\subfigure{\includegraphics[scale=.5]{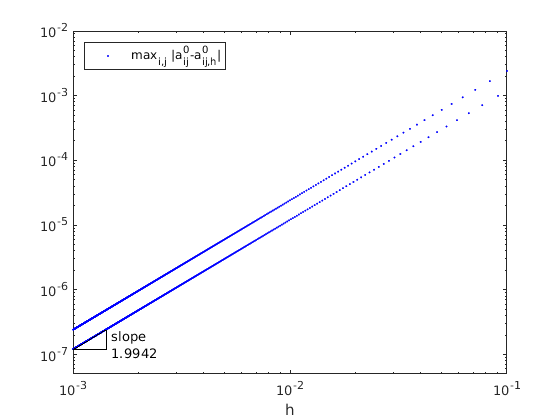}}
}
\vspace{-0.6 cm}
\caption{Approximation error for the invariant measure $m$ (left) and the matrix $A^0$ (right). }
\end{figure}

Figure 2 shows the error in the approximation of $m$ and $A^0$. For the approximation of the invariant measure we observe convergence of order 
\begin{align}\label{order3/2}
\|m-m_h\|_{L^2(Y)} = \mathcal{O}(h^{\frac{3}{2}}),
\end{align}
and superconvergence of order $\mathcal{O}(h^2)$ for $h>0$ when grid points fall on the line $\{y_1=\frac{1}{2}\}$, which is the set along which $\partial_1 m$ possesses a jump. The observed rate of convergence \eqref{order3/2} is consistent with Theorem 3.1. Indeed, we have $m\in H^{\frac{3}{2}-\tilde{\eps}}(Y)$ for any $\tilde{\eps}>0$, and Theorem 3.1 yields
\begin{align*}
\|m-m_h\|_{L^2(Y)}+h\|m-m_h\|_{H^1(Y)}&\lesssim h \inf_{\tilde{v}_h\in  \tilde{M}_h} \|m-(\tilde{v}_h+1)\|_{H^1(Y)}\\
&\lesssim h \left\|m-\mathcal{I}_h m - \int_Y (m- \mathcal{I}_h m)\right\|_{H^1(Y)}
\\&\lesssim h^{\frac{3}{2}-\tilde{\eps}}\|m\|_{H^{\frac{3}{2}-\tilde{\eps}}(Y)} ,
\end{align*}
by making the choice $\tilde{v}_h= \mathcal{I}_h m - \int_Y \mathcal{I}_h m$, and using an interpolation error bound. In connection with the superconvergence we note that $\left. m\right\rvert_{(0,\frac{1}{2})\times (0,1)}\in H^2((0,\frac{1}{2})\times (0,1))$ and $\left. m\right\rvert_{(\frac{1}{2},1)\times (0,1)}\in H^2((\frac{1}{2},1)\times (0,1))$.
 For the approximation of the matrix $A^0$, we observe second-order convergence.

Concerning the approximation of $u_{\eps}$, from Sections 2 and 3.3 we obtain that 
\begin{align*}
\|u_{\eps}-u_0\|_{H^1(\Omega)}+\sum_{k,l=1}^2 \left\|\partial_{kl}^2u_{\eps}-\left(\partial_{kl}^2u_0 + \sum_{i,j=1}^2 \left(\partial_{kl}^2\chi_{ij}\right)\left(\frac{\cdot}{\eps}\right)\partial^2_{ij} u_0\right)\right\|_{L^2(\Omega)}=\mathcal{O}(\sqrt{\eps}),
\end{align*}
where $\chi_{ij}$ ($1\leq i,j\leq 2$) denotes the solution to 
\begin{align*}
\begin{cases}
A:D^2 \chi_{ij} = a_{ij}^0-a_{ij}\quad\text{in }Y,\\
\chi_{ij}\;\text{is $Y$-periodic},\;\int_Y \chi_{ij}=0.
\end{cases}
\end{align*}
Note that since $A$ only depends on $y_1$, we have that
\begin{align*}
\partial_{kl}^2\chi_{ij} (y_1,y_2) = \left\{\begin{aligned}\frac{a_{ij}^0-a_{ij}(y_1)}{a_{11}(y_1)}&\;,\; k=l=1,\\ 0\quad\quad\;\; &\;,\; \text{otherwise.}\end{aligned}\right.
\end{align*}
Therefore, there holds
\begin{align}\label{q10ex}
\begin{split}
\|u_{\eps}&-u_{0}\|_{H^1(\Omega)}^2 +  \left\|\partial_{12}^2u_{\eps}-\partial_{12}^2u_{0} \right\|_{L^2(\Omega)}^2 +\left\|\partial_{22}^2u_{\eps}-\partial_{22}^2u_{0} \right\|_{L^2(\Omega)}^2 \\ &+\left\|\partial_{11}^2u_{\eps}-\left(\partial_{11}^2u_{0} + \sum_{i,j=1}^2 \left(\partial_{11}^2\chi_{ij}\right)\left(\frac{\cdot}{\eps}\right)\partial^2_{ij} u_{0}\right)\right\|_{L^2(\Omega)}^2 =\mathcal{O}(\eps).
\end{split}
\end{align}

For the numerical approximation, we replace $u_{\eps}$ by an $H^2$-conforming finite element approximation on a fine mesh, based on the formulation
\begin{align*}
\text{Find }u_{\eps}\in H:\quad\int_{\Omega}\frac{\mathrm{tr}A\left(\frac{\cdot}{\eps}\right)}{\lvert A\left(\frac{\cdot}{\eps}\right)\rvert^2}\;A\left(\frac{\cdot}{\eps}\right):D^2 u_{\eps}\;\Delta v = \int_{\Omega} \frac{\mathrm{tr}A\left(\frac{\cdot}{\eps}\right)}{\lvert A\left(\frac{\cdot}{\eps}\right)\rvert^2}\;f \Delta v \quad\quad\forall\, v\in H,
\end{align*}
where $H:=H^2(\Omega)\cap H^1_0(\Omega)$. To this end, we use again the HCT element  and set the gradient on the boundary to be the gradient of an $H^1$ approximation by $\mathbb{P}_2$ elements on a fine mesh.

\begin{figure}[H]
\setlength{\abovecaptionskip}{17.5 pt plus 4pt minus 2pt}
\mbox{
\subfigure{\includegraphics[scale=.5]{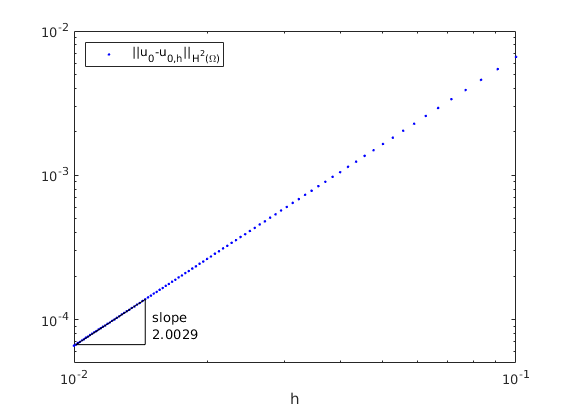}}\,
\subfigure{\includegraphics[scale=.5]{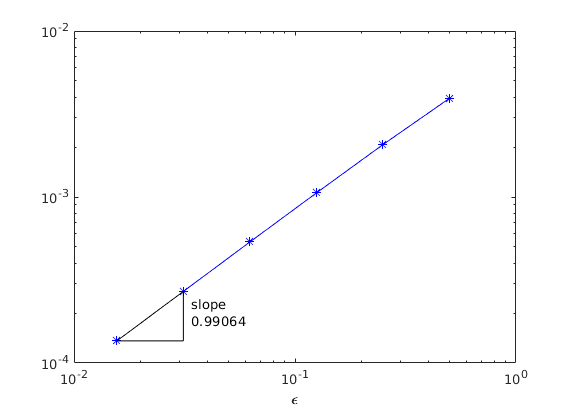}}
}
\vspace{-0.6 cm}
\caption{Approximation error for $u_0$ (left) and the squared error \eqref{q10ex} in the approximation of $u_{\eps}$ for different values of $\eps$ (right). }
\end{figure}

Figure 3 shows the error in the approximation of $u_0$ and we observe second-order convergence. Further, with the exact $u_0$ being available, we can compute the error \eqref{q10ex} for different values of $\eps$; see Figure 3. We observe first-order convergence as $\eps$ tends to zero, as expected from \eqref{q10ex}.

\subsection{Problem with an Unknown $u_0$}

Next, let us consider the problem \eqref{1265} with the same domain $\Omega$ and matrix-valued function $A$ as before, but with the right-hand side given by
\begin{align*}
f:\Omega\rightarrow\R,\quad f(x_1,x_2):=\exp\left( -\frac{1}{\frac{1}{2}-\left(x_1-\frac{1}{2}\right)^2-\left(x_2-\frac{1}{2}\right)^2}\right)  ,
\end{align*}
Note that we are in the situation $(\Omega,A,f)\in \calH^2$. Further, since the right-hand side $f\in H^2(\Omega)$ of the homogenized problem \eqref{15} satisfies $f=0$ at the corners of $\Omega$, the solution $u_0$ to \eqref{15} belongs to the class $H^4(\Omega)$; see \cite[Prop. 2.6]{HOS15}.

As before, we use the scheme presented in Section 3.1 to approximate $m$, $A^0$ and $u_0$. Using the second-order $H^2$ approximation $u_{0,h}$ to $u_0$ obtained as previously described, 
\begin{align*}
\left\| u_0-u_{0,h}\right\|_{H^2(\Omega)} = \mathcal{O}(h^2),
\end{align*}
we have that
\begin{align}\label{q10}
\begin{split}
\|u_{\eps}&-u_{0,h}\|_{H^1(\Omega)}^2 +  \left\|\partial_{12}^2u_{\eps}-\partial_{12}^2u_{0,h} \right\|_{L^2(\Omega)}^2 +\left\|\partial_{22}^2u_{\eps}-\partial_{22}^2u_{0,h} \right\|_{L^2(\Omega)}^2 +\\ &\left\|\partial_{11}^2u_{\eps}-\left(\partial_{11}^2u_{0,h} + \sum_{i,j=1}^2 \left(\partial_{11}^2\chi_{ij}\right)\left(\frac{\cdot}{\eps}\right)\partial^2_{ij} u_{0,h}\right)\right\|_{L^2(\Omega)}^2 =\mathcal{O}(\eps + h^4).
\end{split}
\end{align} 
Figure 4 shows the squared error \eqref{q10} of the approximation of $u_{\eps}$ for different grid sizes and $\eps = \frac{1}{100}$ fixed. We observe fourth-order convergence in $h$ for the squared error as expected from \eqref{q10}.

\begin{figure}[H]
\setlength{\abovecaptionskip}{17.5 pt plus 4pt minus 2pt}
\mbox{
\subfigure{\includegraphics[scale=.5]{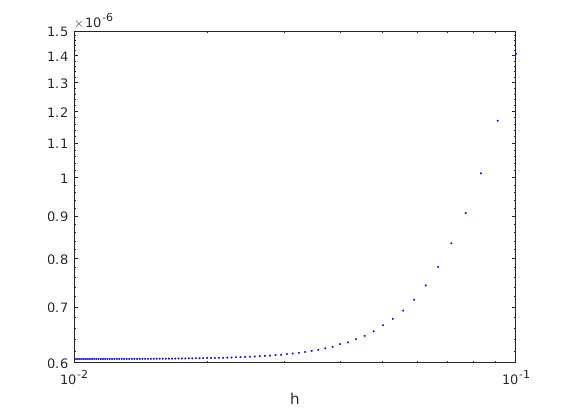}}\,
\subfigure{\includegraphics[scale=.5]{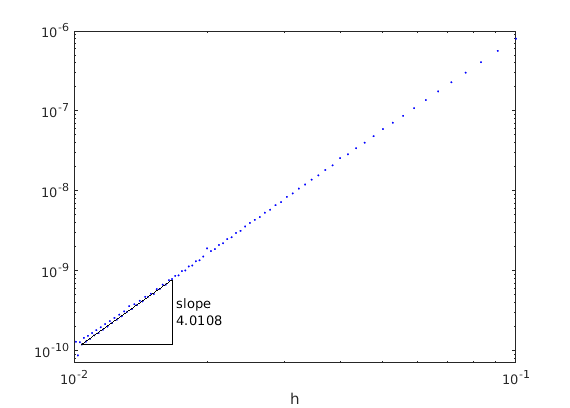}}
}
\vspace{-0.6 cm}
\caption{The squared error \eqref{q10} in the approximation of $u_{\eps}$ for a fixed value, $\eps=\frac{1}{100}$, (left) and the squared error after subtraction of $6.0657\cdot 10^{-7}$ (right), which is approximately the limit of \eqref{q10} in the figure on the left for this fixed value of $\eps$ as $h$ tends to zero.}
\end{figure}

\subsection{Nonuniformly Oscillating Coefficients}

We consider the homogenization problem
\begin{align}\label{NOCproblemnum}
\left\{\begin{aligned}
A\left(\cdot,\frac{\cdot}{\eps}\right):D^2 u_{\eps} &= f\quad\text{in }\Omega,\\
\hfill u_{\eps} &= 0\quad\text{on }\partial\Omega,
\end{aligned}\right.
\end{align}
on the domain 
\begin{align*}
\Omega:=Y=(0,1)^2,
\end{align*}
with the matrix-valued map $A:\Omega\times\R^2\rightarrow \R^{2\times 2}$,
\begin{align*}
(x,y)=((x_1,x_2),(y_1,y_2))\mapsto \begin{pmatrix}
\mathrm{e}^{x_1x_2}+\frac{1}{4}\lvert x\rvert^2 \arcsin\left(\sin^2(\pi y_1)\right)  & 0\\
0 & 2+x_2 \cos(2\pi y_2+x_1)
\end{pmatrix},
\end{align*}
and the right-hand side $f:\Omega\rightarrow\R$ to be specified below. We observe that the matrix-valued function $A$ satisfies \eqref{assnu} with $q=\infty$. Further, note that it is of the form
\begin{align*}
A(x,y)= \mathrm{diag}     \left(a_{11}(x,y_1),a_{22}(x,y_2)\right).
\end{align*}
In this case we know that the homogenized problem is given by
\begin{align}\label{nocexp1}
\left\{\begin{aligned}
A^0:D^2 u_{0} &= f\quad\text{in }\Omega,\\
\hfill u_{0} &= 0\quad\text{on }\partial\Omega,
\end{aligned}\right.
\end{align}
where $A^0:\Omega\rightarrow \R^{2\times 2}$ is given by
\begin{align*}
A^0(x)=\int_Y A(x,\cdot)m(x,\cdot)
\end{align*}
with $m$ being the invariant measure 
\begin{align*}
m:\Omega\times\R^2\rightarrow\R,\quad m(x,y)=\left(\int_0^1 \int_0^1 \frac{\mathrm{d}s\mathrm{d}t}{a_{11}(x,s)a_{22}(x,t)}  \right)^{-1}\frac{1}{a_{11}(x,y_1)a_{22}(x,y_2)};
\end{align*}
see \cite{FO09}. Therefore, we have 
\begin{align*}
a_{ij}^0(x) &= \delta_{ij}\left(\int_0^1 \frac{\mathrm{d}t}{a_{ij}(x,t)}  \right)^{-1},\quad 1\leq i,j\leq 2.
\end{align*}
We consider the right-hand side given by
\begin{align*}
f:\Omega\rightarrow\R,\quad x=(x_1,x_2) \mapsto f(x):=a_{22}^0(x)\, x_1(x_1-1) + a_{11}^0(x)\, x_2(x_2-1).
\end{align*}
Then it is straightforward to check that the exact solution $u_0\in H^2(\Omega)\cap H^1_0(\Omega)$ to the homogenized problem \eqref{nocexp1} is given by
\begin{align*}
u_0:\Omega\rightarrow\R,\quad u_0(x_1,x_2) = \frac{1}{2}x_1(x_1-1)x_2(x_2-1).
\end{align*}
Note that the assumptions of Theorem 3.6 $(iii)$ are satisfied.

For $k>0$ such that $\frac{1}{k}\in \N$, we take a triangulation $\calT_k$ on $\bar{\Omega}$ consisting of nodes $\{(sk,rk)\}_{s,r=0,\dots,1/k}$, and a triangulation $\calT_h$ on $Y$ with grid size $h=\frac{k}{4}$. We use the scheme presented in Section 3.4 to approximate $A^0$ and $u_0$, and we observe second-order convergence; see Figure 5.

For the approximation of $u_{\eps}$, Theorem 3.6 yields

\begin{align*}
\|u_{\eps}-u_0\|_{H^1(\Omega)}+\sum_{k,l=1}^2 \left\|\partial_{kl}^2u_{\eps}-\left(\partial_{kl}^2u_0 + \sum_{i,j=1}^2 \left(\partial_{y_k y_l}^2\chi_{ij}\right)\left(\cdot,\frac{\cdot}{\eps}\right)\partial^2_{ij} u_0\right)\right\|_{L^2(\Omega)}=\mathcal{O}(\sqrt{\eps}),
\end{align*}
where $\chi_{ij}$ ($1\leq i,j\leq 2$) denotes the solution to 
\begin{align*}
\begin{cases}
A(x,y):D_y^2 \chi_{ij}(x,y) = a_{ij}^0(x)-a_{ij}(x,y) ,\quad\,(x,y)\in \Omega\times Y,\\
\chi_{ij}(x,\cdot)\;\text{is $Y$-periodic},\;\int_Y \chi_{ij}(x,\cdot)=0.
\end{cases}
\end{align*}
We observe that we have
\begin{align*}
\partial_{y_k y_l}^2\chi_{ij} (x,y)=\left\{\begin{aligned}\frac{a_{11}^0(x)-a_{11}(x,y_1)}{a_{11}(x,y_1)}&\;,\; i=j=k=l=1,\\ \frac{a_{22}^0(x)-a_{22}(x,y_2)}{a_{22}(x,y_2)}&\;,\; i=j=k=l=2,\\ 0\quad\quad\;\; &\;,\; \text{otherwise.}\end{aligned}\right.
\end{align*}
Therefore, we have that
\begin{align}\label{2liner}
\begin{split}
&\|u_{\eps}-u_0\|_{H^1(\Omega)}^2 + \left\|\partial_{11}^2 u_{\eps}-\left(\partial_{11}^2 u_0 + \left[\partial_{y_1 y_1}^2 \chi_{11}\right]^{\eps}\partial_{11}^2 u_0 \right)  \right\|_{L^2(\Omega)}^2 \\ &+ \|\partial_{12}^2 u_{\eps}-\partial_{12}^2 u_0\|_{L^2(\Omega)}^2 +\left\|\partial_{22}^2 u_{\eps}-\left(\partial_{22}^2 u_0 + \left[\partial_{y_2 y_2}^2 \chi_{22}\right]^{\eps}\partial_{22}^2 u_0 \right)  \right\|_{L^2(\Omega)}^2 = \calO(\eps),
\end{split}
\end{align}
where $\left[\partial_{y_i y_i}^2 \chi_{ii}\right]^{\eps}:= \left(\partial_{y_i y_i}^2 \chi_{ii}\right)\left(\cdot,\frac{\cdot}{\eps}\right) $ for $i\in\{1,2\}$.
For the numerical approximation, we replace $u_{\eps}$ by an $H^2$-conforming finite element method on a fine mesh, based on the formulation
\begin{align*}
\text{Find }u_{\eps}\in H:\quad\int_{\Omega}\frac{\mathrm{tr}A\left(\cdot,\frac{\cdot}{\eps}\right)}{\lvert A\left(\cdot,\frac{\cdot}{\eps}\right)\rvert^2}\;A\left(\cdot,\frac{\cdot}{\eps}\right):D^2 u_{\eps}\;\Delta v = \int_{\Omega} \frac{\mathrm{tr}A\left(\cdot,\frac{\cdot}{\eps}\right)}{\lvert A\left(\cdot,\frac{\cdot}{\eps}\right)\rvert^2}\;f \Delta v \quad\quad\forall\, v\in H,
\end{align*}
where $H:=H^2(\Omega)\cap H^1_0(\Omega)$. To this end, we use again the HCT element  and set the gradient on the boundary to be the gradient of an $H^1$ approximation by $\mathbb{P}_2$ elements on a fine mesh. 

\begin{figure}[H]
\setlength{\abovecaptionskip}{17.5 pt plus 4pt minus 2pt}
\mbox{
\subfigure{\includegraphics[scale=.5]{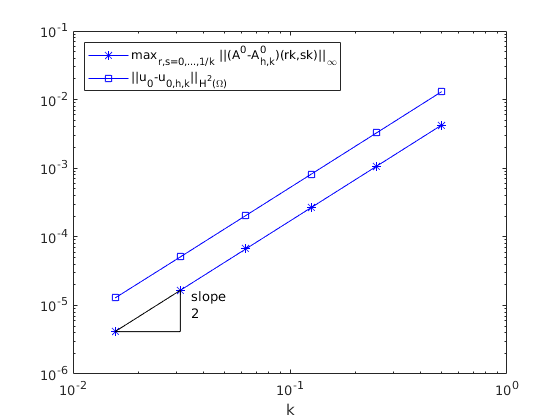}}\,
\subfigure{\includegraphics[scale=.5]{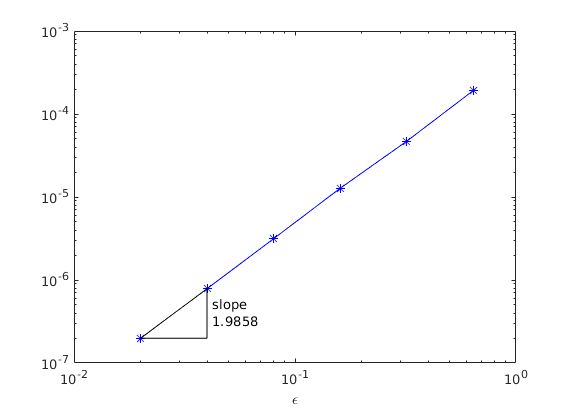}}
}
\vspace{-0.6 cm}
\caption{Approximation error for $A^0$ and $u_0$ for different values of $k$, using $h=\frac{k}{4}$, (left) and the squared error \eqref{2liner} in the approximation of $u_{\eps}$ for different values of $\eps$ (right).}
\end{figure}

Finally, let us consider the problem \eqref{NOCproblemnum} with the same domain $\Omega$ and matrix-valued function $A$ as before, but with the right-hand side given by
\begin{align*}
f:\Omega\rightarrow\R,\quad f(x_1,x_2):=\exp\left( -\frac{1}{\frac{1}{2}-\left(x_1-\frac{1}{2}\right)^2-\left(x_2-\frac{1}{2}\right)^2}\right)  ,
\end{align*}
Note that we are in the situation $(\Omega,A,f)\in \calH$. Further, since the right-hand side $f\in H^2(\Omega)$ of the homogenized problem \eqref{nocexp1} satisfies $f=0$ at the corners of $\Omega$, the solution $u_0$ to \eqref{nocexp1} belongs to the class $H^4(\Omega)$, see \cite[Prop. 2.6]{HOS15}.

Using the second-order $H^2$-conforming approximation $u_{0,k}$ to $u_0$ obtained as previously described (again with $h=\frac{k}{4}$), 
\begin{align*}
\left\| u_0-u_{0,k}\right\|_{H^2(\Omega)} = \mathcal{O}(k^2),
\end{align*}
we have that
\begin{align}\label{q10noc}
\begin{split}
&\|u_{\eps}-u_{0,k}\|_{H^1(\Omega)}^2 + \left\|\partial_{11}^2 u_{\eps}-\partial_{11}^2 u_{0,k} - \left[\partial_{y_1 y_1}^2 \chi_{11}\right]^{\eps}\partial_{11}^2 u_{0,k}   \right\|_{L^2(\Omega)}^2 \\ &+ \|\partial_{12}^2 u_{\eps}-\partial_{12}^2 u_{0,k}\|_{L^2(\Omega)}^2 +\left\|\partial_{22}^2 u_{\eps}-\partial_{22}^2 u_{0,k} - \left[\partial_{y_2 y_2}^2 \chi_{22}\right]^{\eps}\partial_{22}^2 u_{0,k}   \right\|_{L^2(\Omega)}^2 = \calO(\eps+k^4).
\end{split}
\end{align} 
Figure 6 shows the squared error \eqref{q10noc} of the approximation of $u_{\eps}$ for different grid sizes and $\eps = \frac{1}{50}$ fixed. We observe fourth-order convergence in $k$ for the squared error as expected from \eqref{q10noc}.

\begin{figure}[H]
\setlength{\abovecaptionskip}{17.5 pt plus 4pt minus 2pt}
\mbox{
\subfigure{\includegraphics[scale=.5]{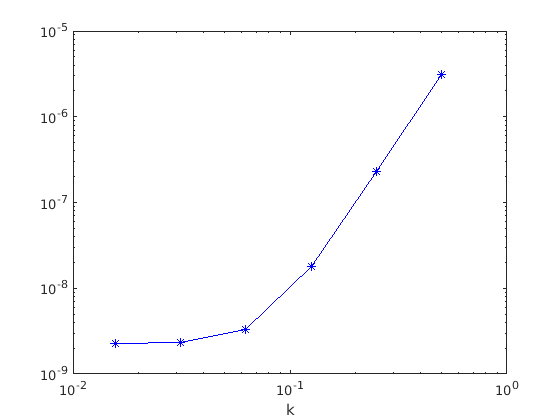}}\,
\subfigure{\includegraphics[scale=.5]{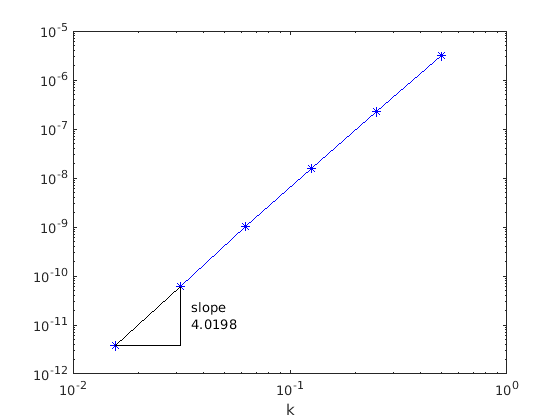}}
}
\vspace{-0.6 cm}
\caption{The squared error \eqref{q10noc} in the approximation of $u_{\eps}$ for a fixed value, $\eps=\frac{1}{50}$, (left) and the squared error after subtraction of $2.2653\cdot 10^{-9}$ (right), which is approximately the limit of \eqref{q10noc} in the figure on the left for this fixed value of $\eps$ as $k$ tends to zero.}
\end{figure}

\section{Conclusion}

In this paper we introduced a scheme for the numerical approximation of elliptic problems in nondivergence-form with rapidly oscillating coefficients on $C^{2,\gamma}$ and polygonal domains, which is based on a $W^{2,p}$ corrector estimate for such problems derived in the first part of this work. 

We proved an optimal-order error bound for a finite element approximation of the corresponding invariant measure using continuous $Y$-periodic piecewise linear basis functions on a shape-regular triangulation of the unit cell $Y$ under weak regularity assumptions on the coefficients. The coefficients are integrated against the so obtained approximation of the invariant measure after piecewise linear interpolation on the mesh to obtain an approximation of the constant coefficient-matrix of the homogenized problem. Using an $H^2$ comparison result for the solution of this perturbed problem, we  eventually obtained an approximation of the solution $u_0$ to the homogenized problem in the $H^2$-norm. In the case of a polygonal domain in two space dimensions, we made use of compatibility conditions for the source term to ensure sufficiently high Sobolev-regularity of $u_0$. 

We obtained an approximation to the solution $u_{\eps}$ of the original problem, i.e., the problem with oscillating coefficients, by making use of the $H^2$ approximation of $u_0$, finite element approximations to second-order derivatives of the corrector functions, as well as an $H^2$ corrector result. A method of successively approximating higher derivatives for the approximation of corrector functions was provided and analyzed. The corrector functions are necessary in order to obtain an approximation of $D^2 u_{\eps}$ whereas the task of approximating $u_{\eps}$ in the $H^1$-norm can be achieved using only an $H^1$ approximation of $u_0$.

Furthermore, we generalized our results to the case of nonuniformly oscillating coefficients, i.e., we derived an analogous corrector result and studied the approximation of the solution $u_0$ to the homogenized problem and the solution $u_{\eps}$ of the $\eps$-dependent problem in this case.

In the final part of the paper, we presented numerical experiments matching the theoretical results for problems with both known and unknown $u_0$, as well as problems with nonuniformly oscillating coefficients. We illustrated the performance of the scheme for the approximation of the invariant measure, the solution $u_0$ to the homogenized problem and the solution $u_{\eps}$ to the problem involving oscillating coefficients for a fixed value of $\eps$.

Future work will focus on weakening of the regularity assumptions on the coefficients and the approximation of fully nonlinear nondivergence-form problems with oscillating coefficients such as the Hamilton--Jacobi--Bellman equation.

\section*{Acknowledgements}

This work was supported by the UK Engineering and Physical Sciences Research Council
[EP/L015811/1].

\bibliographystyle{plain}
\bibliography{ref}

\end{document}